\documentclass[11pt,a4paper,reqno]{amsproc}
\usepackage{amsmath, amsthm, amscd, amsfonts, amssymb, graphicx, color}
\usepackage{subcaption}
\usepackage{enumerate}
\usepackage[bookmarksnumbered, colorlinks, plainpages]{hyperref}
\hypersetup{colorlinks=true,linkcolor=red, anchorcolor=green, citecolor=cyan, urlcolor=red, filecolor=magenta, pdftoolbar=true}

\newtheorem{theorem}{Theorem}[section]
\newtheorem{lemma}[theorem]{Lemma}
\newtheorem{proposition}[theorem]{Proposition}

\theoremstyle{definition}
\newtheorem{definition}[theorem]{Definition}

\theoremstyle{remark}
\newtheorem{remark}[theorem]{Remark}

\numberwithin{equation}{section}

\usepackage[square,numbers]{natbib}

\begin{document}

\setcounter{page}{1}

\title[Caccioppoli-type inequalities associated with the Dunkl-$A$-Laplacian ]
 {Caccioppoli-type inequalities for the Dunkl-$A$-Laplacian and their application to nonexistence result}

\author[Athulya P and Sandeep Kumar Verma ]{Athulya P and Sandeep Kumar Verma}

\address{Department of Mathematics, SRM University-AP, Amaravati, Guntur--522240, India}

\email{athulya.panoli97@gmail.com, sandeep16.iitism@gmail.com}



\begin{abstract} 
For a suitable function $A:\mathbb{R}^n\to \mathbb{R}^n$, we introduce the $A$-Laplacian in the Dunkl framework as $\Delta_{k,A}(u) =\text{div}_k(A(\nabla_ku))$, where $\nabla_k$ is the Dunkl-gradient operator associated with the multiplicity function $k$ and the root system $\mathcal{R}$. We derive the local and global Caccioppoli-type inequality for an element $u$ in the Dunkl-Orlicz-Sobolev space, satisfying the Dunkl-differential inequality   $$ -\Delta_{k, A}(u) \geq b\Phi(u)\chi_{\{u>0\}}. $$ Using the Caccioppoli inequality, we establish a sufficient condition for the nonexistence of a nonzero solution $u$ to the Dunkl-differential inequality.
\end{abstract}
\maketitle

\section{Introduction}
The Caccioppoli inequality plays a central role in the qualitative theory of partial differential equations (PDEs). As a foundational estimate, it is instrumental in studying the existence, uniqueness, and regularity properties of solutions to linear and nonlinear PDEs (see, for example, \cite{Gariepy, Giaquinta, Kalamajska-1}). The classical form of the Caccioppoli inequality asserts that if $u$ is a harmonic or subharmonic function, then the $L^2$ norm of its gradient $\nabla u$ can be controlled by the $L^2$ norm of $u$ itself \cite{Caccioppoli}.   The gradient estimates are essential in a wide range of analytical and physical contexts, particularly when $\nabla u$ represents conservative vector fields. This becomes relevant especially in disciplines like electromagnetism, fluid dynamics, and wave propagation, where gradient-based quantities and the Laplacian operator frequently arise. Motivated by these applications, numerous generalizations of the Caccioppoli inequality have been developed to accommodate more complex PDE settings (see \cite{Chlebicka, Ding, Giaquinta-1, Niu, Seregin, Troianiello, Xing} and references therein).   Among these developments, a prominent example is the study of $A$-harmonic equations, which generalize both harmonic and $p$-harmonic equations by incorporating the function \( A \colon \mathbb{R}^n \to \mathbb{R}^n \).  These equations are central in nonlinear analysis and appear in various models involving non-standard growth conditions. To the best of our knowledge, Caccioppoli-type inequalities have not yet been studied in the context of the Dunkl operator, a differential-difference operator associated with finite reflection groups. The Dunkl operator framework naturally arises in harmonic analysis, special functions, and integrable systems, where reflection symmetries play a role. In this article, we derive the Caccioppoli-type inequalities for the differential-difference operator and discuss the sufficient condition for the nonexistence of the nonzero solution.

\par The Dunkl operators, introduced by Dunkl in 1980s \cite{D1, D2, D3}, and subsequently developed by researchers such as de Jeu, Opdam, and R\"osler \cite{DJ1, Opdam, Rosler-2003, Rosler-2008}, are remarkable generalizations of the classical differential operators. One of the central motivations for studying Dunkl operators arises from the theory of Riemannian symmetric spaces, where spherical functions play a crucial role. These spherical functions can be viewed as multivariable special functions that depend on discrete sets of parameters.  In essence, Dunkl operators are families of commuting differential-difference operators acting on Euclidean space. They are equipped with an additional structure determined by a set of parameters called multiplicities. % which encode the contributions of reflections across hyperplanes 
\cite{Rosler-2008}. On top of that, Dunkl operators also appear in various areas of mathematical physics. Notably, they are relevant in the study of quantum many-body systems of Calogero–Moser–Sutherland type, which are examples of algebraically integrable systems in one dimension. These systems exhibit rich algebraic and geometric structures, further highlighting the deep connections between Dunkl theory and integrable models in physics (see \cite{DV} for an extensive bibliography).
\par The theory of nonexistence is one of the cornerstones in the study of PDEs. Over the years, extensive work has been devoted to developing nonexistence results for nonlinear elliptic equations and systems; see, for instance, \cite{Baldelli, Bidaut, Caristi, Esteban} and the references therein. Several methods have been introduced to establish the existence and nonexistence of solutions to both equations and inequalities in various contexts \cite{Caristi, Esteban, Gossez, Ni}. In \cite{Kalamajska-1}, Ka{\l}amajska et al. developed nonexistence results associated with a Caccioppoli-type inequality for
$$- \Delta_A u \geq \Phi(u).$$
Here, $\Delta_A$ is the $A$-Laplacian, where $A$ is a suitable function on $\mathbb{R}^n$, $\Phi$ is a continuous function on $[0,\infty)$ under appropriate assumptions and $u$ is the Orlicz-Sobolev function. The proof technique by  Ka{\l}amajska et al. is based on techniques developed by Mitidieri and Pohozaev \cite{Mitidieri}. Later, in 2019, Chlebicka et al. \cite{Chlebicka} studied nonlinear PDEs of the form
\begin{align*}
\operatorname{div}(a(x)|\nabla u|^{p-2} \nabla u) \geq b(x)\Phi(u)\chi_{\{u>0\}}, \quad p>1,
\end{align*}
and investigated the existence and nonexistence of solutions, Caccioppoli-type inequalities, and Hardy-type inequalities. Here, the term $|\nabla u|^{p-2}\nabla u$ represents the $p$-Laplacian, which is a particular case of the $A$-Laplacian. The proof technique for the $p$-Laplacian problem is also derived from \cite{Mitidieri}. One can study these partial differential inequalities (PDIs) in a unified way by considering
\begin{align}
-\Delta_{A}u \geq b(x)\Phi(u)\chi_{\{u>0\}}. \label{Intro-1}
\end{align}
The specific cases of inequality \eqref{Intro-1} have already been examined in the literature. For instance, in 1930, Matukuma \cite{Matukuma} proposed the equation
$$\Delta u + \frac{1}{1+|x|^2}u^p = 0 \quad \text{in } \mathbb{R}^3,\quad p > 1,\ u > 0,$$
as a model for the gravitational potential in a globular star cluster. Similarly, the equation
\[
-\Delta u(x) = \phi(r)|u|^{p-2}u \quad \text{in } \mathbb{R}^3,\quad u > 0,
\]
with \( r = \sqrt{x_1^2 + x_2^2} \) and integrability condition \( \int_{\mathbb{R}^3} \phi(r) u(x)^{p-1} \, dx < \infty \), has been used to describe the density profile of elliptic galaxies \cite{Badiale}. 

\par
To the best of our knowledge, nonexistence results based on Caccioppoli-type inequalities have not yet been established in the Dunkl setting. We generalize inequality \eqref{Intro-1} to the Dunkl framework by replacing the classical differential operator with the Dunkl operator. This leads to the formulation  \begin{align*} 
-\Delta_{k,A}u \geq b(x)\Phi(u)\chi_{\{u>0\}}, \end{align*} 
which we refer to as the Dunkl differential inequality (DDI). Here, \( b(x) > 0 \) is continuous function,  \( \chi_{\{u>0\}} \) denotes the indicator function of the set $\{x\in \mathbb{R}^n: u(x)>0\}$, and the operator $\Delta_{k, A}$, is regarded as the $A$-Laplacian associated with the Dunkl operator. In the Dunkl framework, Liouville-type nonexistence results are already known. Gallardo and Godefroy \cite{Gallardo} extended the classical Liouville theorem by proving that any bounded Dunkl harmonic function on $\mathbb{R}^n$ must be constant. A similar result for Dunkl polyharmonic functions was later established in \cite{Ren}. 
In recent years, significant attention has been given to the study of existence and nonexistence problems for differential inequalities involving the Dunkl Laplacian; see, for example, \cite{Jleli-3, Jleli-1, Jleli-2, Jleli}. Inspired by these contributions, we establish a nonexistence result associated with the DDIs. The arguments used in our proofs are mainly based on the techniques developed in \cite{Chlebicka, Kalamajska-1}, which provide a framework for studying nonlinear differential inequalities and present refined versions of earlier methods. In particular, these works build upon the approach of \cite{Mitidieri}, who developed an effective method for proving nonexistence results through integral estimates with carefully chosen test functions. Motivated by these ideas, we adapt their techniques to the Dunkl setting.
First, we establish a Caccioppoli-type inequality for the  Dunkl-$A$-Laplacian, consisting of a generalized differential operator. Specifically, we prove the Dunkl-Caccioppoli-type inequality for functions \( u \) satisfying the DDI. The Dunkl-Caccioppoli-type inequality reads:\\

\noindent \textbf{Theorem A} (Theorem \ref{thrm-1}):  
Let \( u \in W_k^{1,\Lambda}(\mathbb{R}^n) \) be a $G$-invariant,  nonnegative solution to the DDI,
\[
-\Delta_{k,A}u \geq b(x)\Phi(u)\chi_{\{u>0\}}.
\]
 Then $u$ satisfies the Dunkl-Caccioppoli-type inequality
\begin{align*}
\int_{\{u>0\}} &\left( b(x)\Phi(u(x)) + 
\beta\frac{\Lambda(|\nabla_k u(x)|)}{g(u(x))} \right) 
\psi(u(x)) \phi(x) w_k(x)\,dx \\
&\leq  \int_{\{\nabla_k u \neq 0,\ u>0,\ \phi \neq 0\}} \Lambda\left( \frac{|\nabla_k \phi(x)|}{\phi(x)}g(u(x)) \right) \frac{\psi(u(x))}{g(u(x))}\phi(x) w_k(x)\,dx,
\end{align*}
for every nonnegative, compactly supported Lipschitz function \( \phi \) such that
\[
\int_{\{\phi \neq 0,\ \nabla_k u \neq 0\}} \Lambda\left( \frac{|\nabla_k \phi(x)|}{\phi(x)} \right) \phi(x) w_k(x)\,dx < \infty,
\]
where $\beta = C_\psi-1$, and the functions $ \Lambda,u,b,g,\Phi, \psi$, and $\phi$ satisfy the assumptions of Proposition \ref{loc.est}. Here,
$W_k^{1,\Lambda}(\mathbb{R}^n)$ denotes Dunkl-Orlicz-Sobolev space (see Definition \ref{DOS-space}).

\medskip The Dunkl-Caccioppoli-type inequality is obtained by first proving the following local estimate (see Proposition \ref{loc.est}) 
\begin{align*}
&\int_{\{0<u<l\}} \left( b(x)\Phi(u(x)) + 
\beta\frac{\Lambda(|\nabla_k u(x)|)}{g(u(x))} \right) 
\psi(u(x)) \phi(x) w_k(x)\,dx \\
&\leq \int_{\{\nabla_k u \neq 0,\ 0<u<l-\delta,\ \phi \neq 0\}} 
\Lambda\left( \frac{|\nabla_k \phi(x)|}{\phi(x)}g(u(x)+\delta) \right)
\frac{\psi(u(x)+\delta)}{g(u(x)+\delta)} \phi(x) w_k(x)\,dx \\
&\quad +\psi(l) \int_{\{ u> l\}} 
    \Bigl(
    B(|\nabla_ku(x)|) \langle \nabla_ku(x), \nabla_k\phi(x)\rangle_{\mathbb{R}^n}- b(x) \Phi(u(x))\phi(x)\Bigr)  w_k(x)\,dx.
\end{align*}Then letting the limit \( l \to \infty \). The local inequality is established via two lemmas. In the first lemma 
( Lemma \ref{Lemma-1-main}), we prove:
\begin{align}
&\int_{0<u\leq l-\delta} \left( b(x)\Phi(u(x)) + 
\beta\frac{\Lambda(|\nabla_k u(x)|)}{g(u(x)+\delta)} \right) 
\psi(u(x)+\delta) \phi(x) w_k(x)\,dx \notag \\
&\leq \int_{\{\nabla_k u \neq 0,\ 0<u<l-\delta,\ \phi \neq 0\}} 
\Lambda\left( \frac{|\nabla_k \phi(x)|}{\phi(x)}g(u(x)+\delta) \right)
\frac{\psi(u(x)+\delta)}{g(u(x)+\delta)} \phi(x) w_k(x)\,dx \notag \\
& + \psi(l) \int_{\{ u > l-\delta \}} \left(
B(|\nabla_k u|)\langle \nabla_k u, \nabla_k \phi \rangle 
- b(x)\Phi(u(x))\phi(x) \right) w_k(x)\,dx. \label{intro-2}
\end{align}

\noindent In the second lemma  (Lemma \ref{Lemma-2-main}), we complete the proof by taking the limit \( \delta \searrow 0 \) in \eqref{intro-2}, using the Lebesgue dominated convergence theorem. With the help of the Dunkl-Caccioppoli-type inequality, we prove the following nonexistence result.\\

\noindent \textbf{Theorem B} (Theorem \ref{Ex-thrm}): 
Let $\Lambda$ be a non-decreasing function satisfying the $M$-condition, and define
\begin{align*}
    \mathfrak{F}(l)=\frac{\Lambda\left(\frac{1}{l}\right) l^n }{\|b\|_{L^\infty(B(x_0,2l))}^{\frac{t}{1-t}} (F^*\circ \Lambda)^{-1} \left( \frac{\Lambda^{\frac{1}{1-t}} \left( \frac{1}{l}\right)}{F^*\circ \Lambda \left( \frac{1}{l}\right)}\right) }, \text{ for $l>0$ and $0<t<1$}. 
\end{align*}

Then the following assertions hold:
\begin{enumerate}[$(i)$]
    \item If $u\in W_{k,\text{loc}}^{1,\Lambda}$ is nonnegative $G$-invariant function and satisfies the DDI, then $\Phi(u)$ is locally integrable and there exists a constant $C>0$ such that 
    \begin{align*}
        \int_{B(x_0,l)\cap \{u>0 \}} b(x)\Phi(u)(x) w_k(x)\,dx \leq C \mathfrak{F}(l),
    \end{align*} for all $x_0\in \mathbb{R}^n$ and $l>0$. 
    \item If $\lim\limits_{l\to\infty} \mathfrak{F}(l)=0$, then there does not exists nonnegative, nonzero solution $u\in W_{k,\text{loc}}^{1,\Lambda}$ satisfying the DDI.
\end{enumerate}    
The rest of the paper is organized as follows: Section \ref{S:2} is devoted to reviewing the fundamentals of the Dunkl operator and some preliminary definitions and results for the main result. In Section \ref{S:3}, we define the Dunkl $A$-Laplacian and prove the Dunkl-Caccioppoli-type inequalities. Section \ref{S:4} concentrates on the nonexistence result associated with the DDI. Moreover, we modify the derived Dunkl-Caccioppoli-type inequality and establish a suitable function to control the nonexistence result of the DDI. 

\section{Preliminaries} \label{S:2}
In this section, we provide a brief overview of essential concepts and tools from Dunkl theory, as well as some fundamental definitions that will be used throughout this paper. For a more detailed study of Dunkl theory, see \cite{DJ1, D1, D2, D3, Rosler-2003} and references therein.

\subsection{Dunkl Operators}
Let $\mathbb{R}^n$ be the standard $n$-dimensional Euclidean space equipped with the inner product $\langle x, y \rangle_{\mathbb{R}^n} = \sum_{j=1}^n x(j)y(j)$ for some $x=(x(1), x(2),\cdots, x(n))$ and $y=(y(1), y(2), \cdots , y(n)) \in \mathbb{R}^n$. A finite set $\mathcal{R} \subset \mathbb{R}^n \setminus \{0\}$ is called a root system if it satisfies the following two properties:
\begin{itemize}
    \item For each $\alpha \in \mathcal{R}$, the set $\mathcal{R} \cap \mathbb{R}\alpha = \{\pm \alpha\}$,
    \item $\sigma_\alpha(\mathcal{R}) = \mathcal{R}$ for every $\alpha \in \mathcal{R}$, where $\sigma_\alpha(x)=x-2\frac{\langle x, \alpha \rangle_{\mathbb{R}^n}}{|\alpha|^2} \alpha$.
\end{itemize}
Throughout the paper, we will assume that  $|\alpha|^2 = 2$ for all $\alpha \in \mathcal{R}$. The root system can be decomposed into a disjoint union $\mathcal{R} = \mathcal{R}_+ \cup \mathcal{R}_-$, 
where $\mathcal{R}_+$ and $\mathcal{R}_-$ denote the sets of positive and negative roots, respectively. 
This decomposition is defined by choosing a vector $y \in \mathbb{R}^n$ such that $y \notin \alpha^\perp$ 
for every $\alpha \in \mathcal{R}$, and setting $\mathcal{R}_+ = \{ \alpha \in \mathcal{R} : \langle \alpha, y \rangle_{\mathbb{R}^n} > 0 \}, 
\qquad 
\mathcal{R}_- = -\mathcal{R}_+.$
The group $G$, generated by the reflections $\sigma_\alpha$ for $\alpha \in \mathcal{R}$, is called the reflection group (or Weyl group) associated with the root system $\mathcal{R}$. A function $k: \mathcal{R} \rightarrow \mathbb{C}$ is called a multiplicity function if it is $G$-invariant, i.e., $k(w\alpha) = k(\alpha)$ for all $w \in G$ and $\alpha \in \mathcal{R}$. We refer the readers to \cite{Hu, RKAN} for more details on the theory of root systems and reflection groups.\\
For a nonnegative multiplicity function $k$,  the weight function $w_k: \mathbb{R}^n \to \mathbb{R}$  associated with the Dunkl operator is
\begin{align*}
    w_k(x) = \prod_{\alpha \in \mathcal{R}_+} |\langle \alpha, x \rangle_{\mathbb{R}^n}|^{2k(\alpha)},
\end{align*}
which  is invariant under the group $G$ and homogeneous of degree $2\gamma$, where $\gamma = \sum\limits_{\alpha \in \mathcal{R}_+} k(\alpha)$.
Given a fixed positive subsystem $\mathcal{R}_+$ and a multiplicity function $k$, for $f \in \mathcal{C}^\prime(\mathbb{R}^n),$
the Dunkl operator \cite{D1} in the direction of a vector $\xi \in \mathbb{R}^n$ is defined as  
\begin{align} \label{Dunkl operator}
    \mathcal{T}_{\xi}f(x) = \partial_{\xi}f(x) + \sum_{\alpha \in \mathcal{R}_+} k(\alpha) \langle \alpha, \xi \rangle_{\mathbb{R}^n} \frac{f(x) - f(\sigma_\alpha(x))}{\langle \alpha, x \rangle_{\mathbb{R}^n}}.
\end{align}
When $k \equiv 0$, it reduces to the classical directional derivative in the direction of $\xi$. For notational convenience, we denote $ \mathcal{T}_{e_j}=\mathcal{T}_j $, where $\{e_1, \dots, e_n\}$ is the standard orthonormal basis of $\mathbb{R}^n$. Then, the associated gradient operator in the Dunkl framework is defined by
$$ \nabla_k =(\mathcal{T}_1, \mathcal{T}_2,\cdots,\mathcal{T}_n)$$
and the Dunkl Laplacian $\Delta_k$ \cite{D1} is given as the sum of the squares of the Dunkl operators:
$$\Delta_k = \sum_{j=1}^n \mathcal{T}_j^2.$$
It admits the representation
\begin{align*}
    \Delta_kf(x) = \Delta f(x) + 2\sum_{\alpha \in \mathcal{R}_+} k(\alpha) \left( \frac{\langle \nabla f(x), \alpha \rangle_{\mathbb{R}^n}}{\langle \alpha, x \rangle_{\mathbb{R}^n}} - \frac{f(x) - f(\sigma_\alpha(x))}{\langle \alpha, x \rangle^2_{\mathbb{R}^n}} \right),
\end{align*}
where $\Delta$ and $\nabla$ are the standard Laplacian and gradient on $\mathbb{R}^n$ respectively. Several properties of partial derivatives carry forward to Dunkl operators \cite{D2, D3}. The following lemma is a kind of chain rule for the Dunkl operator.
\begin{lemma}\label{Chain Rule}
    Let $\Psi \in \mathcal{C}^1(\mathbb{R})$ and $u \in \mathcal{C}'(\mathbb{R}^n)$ be invariant under the reflection group $G$. Consider the composition $\Psi \circ u$ defined by $\Psi \circ u(x) = \Psi(u(x))$.  
    Then, for each $j = 1, \dots, n$, we have
    \begin{align}
        \mathcal{T}_j  \Psi \circ u(x)  = \Psi'(u(x)) \, \mathcal{T}_j u(x). \label{chain rule-1}
    \end{align}
\end{lemma}

\begin{proof}
    Using \eqref{Dunkl operator}, we obtain
    \begin{align*}
         \mathcal{T}_j\Psi \circ u(x) = \partial_j \Psi \circ u(x)  + \sum_{\alpha\in \mathcal{R}_+} k(\alpha) \alpha(j) \frac{\Psi \circ u(x) -\Psi \circ u(\sigma_\alpha(x))}{\langle \alpha, x \rangle_{\mathbb{R}^n}}.
    \end{align*} Since $u$ is $G$-invariant, we have $u(\sigma_\alpha(x)) = u(x)$ for all $\alpha \in \mathcal{R}_+$,  
    which implies that $ \mathcal{T}_j\Psi \circ u(x) = \Psi^\prime(u(x)) \mathcal{T}_ju(x).$  
\end{proof}
\begin{remark}
  A result similar to \eqref{chain rule-1} can be obtained by relaxing the $G$-invariance condition on $u$, provided that $\Psi$ is a linear function.
\end{remark}
Now, we list a few definitions for further reference. 
For $1 \leq p < \infty$, we define the weighted $L^p$-space:
\begin{align*}
    L_k^p(\mathbb{R}^n) = \left\{ f : \mathbb{R}^n \to \mathbb{C} \ \middle| \ \int_{\mathbb{R}^n} |f(x)|^p w_k(x) \, dx < \infty \right\},
\end{align*}
where $dx$ denotes the Lebesgue measure. For $p = \infty$, $ L_k^\infty(\mathbb{R}^n) = L^\infty(\mathbb{R}^n)$.

\begin{definition}
    Let $1\leq p \leq \infty$, $m \in \mathbb{N}$ and $a=(a_1, a_2, \cdots a_n)\in \mathbb{N}_0^n$. The Sobolev space associated with the Dunkl operator \cite{Velichu} is defined as
    \begin{align*}
        W_k^{m,p}(\mathbb{R}^n) = \{ f\in L_k^p(\mathbb{R}^n): \mathfrak{D}_k^a f \in L_k^p(\mathbb{R}^n),\text{ for all $a$ such that $ a_1+a_2+\cdots+a_n \leq m$ }  \},
    \end{align*} where $\mathfrak{D}_k^a f= \mathcal{T}_1^{a_1}\mathcal{T}_2^{a_2}\cdots\mathcal{T}_n^{a_n}f$. The local Dunkl-Sobolev space  $W_{k, \text{loc}}^{m,p}$ is given by  
    \begin{align*}
        W_{k, \text{loc}}^{m,p}(\mathbb{R}^n) = \{ f \in L_k^p(\mathbb{R}^n):  f \in W_{k}^{m,p}(\Omega), \text{ for all compact subset } \Omega \subset \mathbb{R}^n
        \}.
    \end{align*}
\end{definition}
 
\noindent For an arbitrary function $f \in W_{k,\mathrm{loc}}^{m,p}(\mathbb{R}^n)$,  
we define its point-wise value in the average sense by  
\begin{align} \label{point-wise}
    f(x) := \limsup_{r \to 0} \frac{1}{w_k(B(x,r))} 
    \int_{B(x,r)} f(y) \, w_k(y) \, dy,
\end{align}
where $B(x,r)$ denotes the ball of radius $r$ centered at $x$ and $w_k(B(x,r))= \int_{B(x,r)}w_k(x)\,dx$.
\begin{lemma}\cite{Kalamajska-1} \label{gradient-msr}
Let \( u \in W^{1,1}_{\mathrm{loc}}(\mathbb{R}^n) \). Then for each \( s \in \mathbb{R} \),
\[
\{ x \in \mathbb{R}^n : u(x) = s \} \subseteq \{ x \in \mathbb{R}^n : \nabla u(x) = 0 \} \cup N_s,
\]
where $N_s$ is a null set, and $u(x)$ is to be  understood in the sense of \eqref{point-wise}.
\end{lemma}

\begin{definition}
    A function $U:[0,\infty) \to [0,\infty)$ is an  \textbf{$N$-function} if it is convex and $\lim\limits_{s\to 0}\frac{U(s)}{s}=0$ and $\lim\limits_{s\to \infty}\frac{s}{U(s)}=0$.
\end{definition}

\begin{definition}
    A function $U:[0,\infty) \to [0,\infty)$ satisfies the \textbf{$M$-condition}, if there exists a constant $M_U>0$ such that for every $s_1,s_2\in [0,\infty)$  we have 
    \begin{align*}
        U(s_1s_2) \leq M_U U(s_1)U(s_2),
    \end{align*} where the constant $M_U$ is independent of the points $s_1$ and $s_2$.
\end{definition}
\begin{lemma} \label{N-function}\cite{Kalamajska-1} Let $U$ be an $N$-function, then the following holds
\begin{align*}
    \frac{U(s_1)}{s_1}s_2 \leq U(s_1)+U(s_2), \text{ for every $s_1,s_2>0.$}
\end{align*}
    
\end{lemma}

\section{The Dunkl-\texorpdfstring{$A$}{A}-Laplacian and Caccioppoli-type inequality} \label{S:3}

In this section, we introduce the notion of $A$-Laplacian within the Dunkl framework and establish a corresponding Caccioppoli-type inequality. Various forms of Caccioppoli-type inequalities have been explored in the literature, including in abstract and generalized settings \cite{Ding,Niu}. Notably, in 2019, Chlebicka et al.\ \cite{Chlebicka} investigated nonlinear partial differential inequalities of the form
\begin{align*}
- \operatorname{div}(a(x)|\nabla u|^{p-2} \nabla u) \geq b(x)\Phi(u)\chi_{{u>0}}, \quad p>1,
\end{align*}
and derived Caccioppoli-type as well as Hardy-type inequalities for the solutions to weighted 
$p$-harmonic problems in the Euclidean setting. Motivated by their work, we extend the framework to the Dunkl setting by replacing the classical gradient with the Dunkl gradient, generalizing the $p$-Laplacian to the more flexible $A$-Laplacian operator, and restricting the function $a(x)$ to be $1$.
\par We define the Dunkl-Orlicz-Sobolev space  associated with an $N$-function  $U:[0,\infty)\to [0,\infty)$. 
\begin{definition} \label{DOS-space} Dunkl-Orlicz-Sobolev space $W_k^{1,U}(\mathbb{R}^n)$ is the completion of the space 
\begin{align*}
    \{ f\in \mathcal{C}^\infty(\mathbb{R}^n): \|f\|_{W_k^{1,U}(\mathbb{R}^n)}:= \|f\|_{L_{k,U}(\mathbb{R}^n)} + \|\nabla_k f\|_{L_{k,U}(\mathbb{R}^n)} < \infty\},
\end{align*} where 
\begin{align*}
    \|f\|_{L_{k,U}(\mathbb{R}^n)} = \inf \left\{ C>0: \int_{\mathbb{R}^n} U\left(\frac{|f(x)|}{C} \right)w_k(x)\,dx \leq 1 \right\} 
\end{align*} is the Dunkl-Orlicz norm.
\end{definition}
Conventionally, we assume that $\inf\emptyset = +\infty$. The local Dunkl-Orlicz-Sobolev space is denoted by  $W_{k, \text{loc}}^{1,U}(\mathbb{R}^n)$ and defined by 
\begin{align*}
 W_{k, \text{loc}}^{1,U}(\mathbb{R}^n)   =\{u\in \mathcal{C}^\infty(\mathbb{R}^n): u\phi\in W_k^{1,U}(\mathbb{R}^n) \text{ for all } \phi \in \mathcal{C}_c^\infty(\mathbb{R}^n)\}.
\end{align*}
The space $ L_{k, U}(\mathbb{R}^n)$  can be regarded as a weighted counterpart of the classical Orlicz space $L_{U}(\mathbb{R}^n)$, where the weight arises naturally from the Dunkl framework \cite{Guliyev}. Moreover, when the multiplicity function $ k \equiv 0$, the Dunkl-Orlicz-Sobolev space coincides with the classical Orlicz-Sobolev space. Further results on classical Orlicz-Sobolev spaces can be found in \cite{Donaldson}. The Orlicz–Sobolev space provides a natural framework for defining the $A$-Laplacian in the classical setting \cite{Kalamajska-1, Skrzypczak}. It is worth noting that Orlicz spaces are also used in the formulation of the $A$-Laplacian in the differential forms \cite{Niu}.

% \begin{remark}
%     The $\mathcal{C}_c^\infty(\mathbb{R}^n)$ is dense in $W_{k}^{1,U}(\mathbb{R}^n)$ for any $N$-function $U$. The proof will follow from the classical method \cite{} . 
% \end{remark}

\par To this end, we introduce the Dunkl-$A$-Laplacian as  $\Delta_{k,A}(u)= \text{div}_k(A(\nabla_ku))$, under the assumptions that 
\begin{enumerate}[$(i)$]
    \item $A:\mathbb{R}^n\to\mathbb{R}^n$ is a function of the form $A(x)=B(|x|)x,$ where $B$ is a nonnegative continuous function on $[0,\infty)$.
    \item  For $s\in [0,\infty)$,  $\Lambda(s):=B(s)s^2$ is an $N$-function.
\end{enumerate}

Let $u\in W_{k, \text{loc}}^{1,\Lambda}(\mathbb{R}^n)$ and $v\in W_{k}^{1,\Lambda}(\mathbb{R}^n)$ be compactly supported nonnegative functions. The operator $\Delta_{k,A} $ is well defined by the formula
\begin{align*}
    \langle \Delta_{k,A}u,v\rangle & =  \langle\text{div}_k(A(\nabla_ku)), v \rangle = \langle \text{div}_k \left(B(|\nabla_ku|)\nabla_ku \right),v \rangle\\
    &=  -\int_{|\nabla_ku|\neq 0} B(|\nabla_ku|)\langle \nabla_ku, \nabla_kv \rangle_{\mathbb{R}^n} w_k(x)\,dx.
\end{align*}
The aforementioned identity follows from the duality property of the weighted Orlicz space and the Dunkl–Orlicz–Sobolev space. For more details on the duality of Orlicz spaces, see \cite{Harjulehto}.
The Dunkl-$A$-Laplacian serves as a unifying generalization that encompasses several differential operators as special cases, depending on the choice of the function $A$ and the parameter $k$. For example, when $A(x) = |x|^{p-2}x$, $p>1$, and $k\neq 0$, it coincides with the Dunkl-$p$-Laplacian, extending the classical, well-known nonlinear $p$-Laplacian to the Dunkl framework.  In the particular case where $A(x)=x$, i.e., $A$  is the identity map, the operator becomes the classical Dunkl Laplacian. Similarly, when the multiplicity function $k$ is identically zero and the choice of $A$ functions, the Dunkl-$A$-Laplacian reduces to the classical $p$-Laplacian and Laplacian.

\par Now we are in a position to define the Dunkl-differential inequality in the following way.
\begin{definition}
    Let $u\in  W_{k, \text{loc}}^{1,\Lambda}(\mathbb{R}^n)$ be nonnegative function,  
    $\Phi:(0,\infty) \to (0,\infty)$ and  $b:\mathbb{R}^n \to (0,\infty)$ are continuous functions satisfying $b\Phi(u)\chi_{\{u>0\}} \in L_{loc}^1(\mathbb{R}^n)$. The DDI reads 
    \begin{align*}
      -\Delta_{k,A}u \geq b\Phi(u)\chi_{\{u>0\}}.
    \end{align*}For every nonnegative compactly supported function $v\in W_{k}^{1,\Lambda}(\mathbb{R}^n), $ we have \begin{align} \label{W.S of Laplace}
         \langle - \Delta_{k,A}u,v\rangle& 
          \geq \int_{\{u>0\}}b(x)\Phi(u(x))v(x)w_k(x)\,dx. 
    \end{align}
\end{definition}
From now on, we say that $u$ satisfies the DDI if $u \in W_{k,\text{loc}}^{1,\Lambda}$ and the inequality \eqref{W.S of Laplace} holds for all nonnegative compactly supported functions $v \in W_{k,\text{loc}}^{1,\Lambda}$. We state our results in the sequel. The proof of the Dunkl-Caccioppoli-type inequality relies on techniques developed by Ka{\l}amajska et al. \cite{Kalamajska-1}. As a first step, we establish a local estimate of the Dunkl–Caccioppoli-type inequality.

\begin{proposition}\label{loc.est}(Local estimate) Assume that 
\begin{enumerate}[$(i)$]
    \item $\Lambda$ is a non-decreasing function which satisfies $M$-condition.
    \item $(\psi,g)$ is a pair of continuous functions $(\psi,g):(0,\infty)\times (0,\infty)\to (0,\infty)\times (0,\infty)$, where $\psi$ is  non-increasing and satisfies the following compatibility conditions:
     \begin{enumerate}[(a)] \item There exists a constant $C_\psi>1$ independent of $s,$ such that  \begin{align*}
     g(s)\psi^\prime(s) \leq -C_\psi\psi(s) \text{ \quad for a.e.  $s \in (0,\infty)$}.
    \end{align*}
    \item For $s\in (0,\infty)$, the function $\Theta(s)=\frac{\Lambda(g(s))}{g(s)}\psi(s)$
    is bounded on every interval  $(0,l]$ with $l>0$, and $ \zeta(s)= \frac{\psi(s)}{g(s)} $ is  non-increasing. 
     \end{enumerate}
\end{enumerate}Let $u$ be a $G$-invariant, nonnegative solution of the DDI. Then, for any $l>0$, the inequality 
\begin{align}
   & \int_{\{0<u<l\}} \left( b(x)\Phi(u(x)) + 
    \beta\frac{\Lambda(|\nabla_ku(x)|)}{g(u(x))}
    \right) \psi(u(x)) \phi(x)w_k(x)\,dx \notag \\
    & \leq  \int_{\{\nabla_ku \neq 0,0<u<l, \phi \neq0\}} \Lambda\left( \frac{|\nabla_k \phi(x)|}{\phi(x)}g(u(x)) \right) \frac{\psi(u(x))}{g(u(x))}\phi(x) w_k(x)\,dx + C_k(l) \label{local ineq}
\end{align} holds for every nonnegative Lipschitz function $\phi$ with compact support such that the integral $$\int_{\{\phi\neq0\}} \Lambda\left( \frac{|\nabla_k\phi(x)|}{\phi(x)}\right)\phi(x)w_k(x)\,dx < \infty,$$
 where 
 \begin{align}
     C_k(l) = \psi(l) \int_{\{ u> l\}} 
    \Bigl(
    B(|\nabla_ku(x)|) \langle \nabla_ku(x), \nabla_k\phi(x)\rangle_{\mathbb{R}^n}- b(x) \Phi(u(x))\phi(x)\Bigr)  w_k(x)\,dx 
 \label{C_k}
 \end{align}
 and $\beta=  C_\psi-1$.  
\end{proposition}
 We prove the inequality \eqref{local ineq} in  Lemma \ref{Lemma-2-main}. The following lemma proves the inequality \eqref{local ineq} with a $\delta$ factor.
 % Finally, we deduce the inequality by letting  $\delta \searrow 0$.

 \begin{lemma} \label{Lemma-1-main}
     Let the functions $ \Lambda,u,b,g,\Phi, \psi$ and $\phi$  as in the Proposition \ref{loc.est}. Then for any $l>0$, we have the inequality
     \begin{align}
     &\int_{0<u\leq l-\delta} \left( b(x)\Phi(u(x)) + 
    \beta\frac{\Lambda(|\nabla_ku(x)|)}{g(u(x)+\delta)}
    \right) \psi(u(x)+\delta) \phi(x)w_k(x)\,dx \notag  \\
    & \leq   \int_{\{\nabla_ku \neq 0,0<u<l-\delta, \phi \neq0\}} \Lambda\left( \frac{|\nabla_k \phi(x)|}{\phi(x)}g(u(x)+\delta) \right) \frac{\psi(u(x)+\delta)}{g(u(x)+\delta)}\phi(x) w_k(x)\,dx \notag\\
    & +\psi(l) \int_{\{ u > l-\delta \}} \left(
B(|\nabla_k u|)\langle \nabla_k u, \nabla_k \phi \rangle 
- b(x)\Phi(u(x))\phi(x) \right) w_k(x)\,dx,\label{Lemma-1}
     \end{align}  for  $0<\delta<l$, where $\beta=C_\psi-1$.
 \end{lemma}

\begin{proof}
    Let $l>0$ be fixed. Define $u_{\delta,l}(x)=\min\{ u(x)+\delta,l\}$ and $Q(x)=\psi(u_{\delta,l}(x))\phi(x)$, where $0<\delta<l$. We observe that $Q\in W_{k,\text{loc}}^{1,\Lambda}$ and 
    \begin{align}
        I:&=\int_{\{u>0\}}b(x)\Phi(u(x))Q(x)w_k(x)\,dx \notag \\
          & = I_0+  I_1, \label{ineq-1}
    \end{align}
where 
\begin{align*}
    I_0 & = \int_{\{0<u<l-\delta\}} b(x)\Phi(u(x))\psi(u(x)+\delta)\phi(x)w_k(x)\,dx \\
    I_1 & =\psi(l)\int_{\{u>l-\delta\}} b(x)\Phi(u(x))\phi(x)w_k(x)\,dx. \hspace{2.3cm}
\end{align*} Using the inequality \eqref{W.S of Laplace} we have the estimates
\begin{align}
     & I \leq  \langle - \Delta_{k,A}u,Q\rangle 
         =  \int_{\{\nabla_ku \neq 0\} } B(|\nabla_ku(x)|)\langle \nabla_ku(x), \nabla_kQ(x) \rangle_{\mathbb{R}^n} w_k(x)\,dx \notag\\
      &  = \int_{\{\nabla_ku \neq 0,0<u<l-\delta\}} B(|\nabla_ku(x)|)\langle \nabla_ku(x), \phi(x)\nabla_k\psi(u(x)+\delta)+\psi(u(x)+\delta)\nabla_k \phi(x)\rangle_{\mathbb{R}^n} w_k(x)\,dx \notag \\
      & +  \psi(l)\int_{\{\nabla_ku \neq 0,u>l-\delta\}} B(|\nabla_ku|(x))\langle \nabla_ku(x),\nabla_k \phi(x)\rangle_{\mathbb{R}^n} w_k(x)\,dx \notag
      \\
      &=   I_2 + I_3+I_4, \label{ineq-2}
\end{align} where 
\begin{align*}
    I_2 &=  \int_{\{\nabla_ku \neq 0,0<u<l-\delta\}} B(|\nabla_ku(x)|)\langle \nabla_ku(x), \nabla_k\psi(u(x)+\delta)\rangle_{\mathbb{R}^n} \phi(x) w_k(x)\,dx \\
    I_3 & =  \int_{\{\nabla_ku \neq 0,0<u<l-\delta\}} B(|\nabla_ku(x)|)\langle \nabla_ku(x),   \nabla_k \phi(x) \rangle_{\mathbb{R}^n} \psi(u(x)+\delta) w_k(x)\,dx\\
    I_4 &=  \psi(l)\int_{\{\nabla_ku \neq 0,u>l-\delta\}} B(|\nabla_ku(x)|)\langle \nabla_ku(x),\nabla_k \phi(x)\rangle_{\mathbb{R}^n} w_k(x)\,dx.
\end{align*}
Now consider the integral $I_2$. Using Lemma \ref{Chain Rule} and the compatibility condition of $(\psi,g)$ defined in Proposition \ref{loc.est}, we see that  
\begin{align}
     I_2 &=  \int_{\{\nabla_ku \neq 0,0<u<l-\delta\}} B(|\nabla_ku(x)|)\langle \nabla_ku(x), \nabla_k\psi(u(x)+\delta)\rangle_{\mathbb{R}^n} \phi(x) w_k(x)\,dx \notag \\
     & =  \int_{\{\nabla_ku \neq 0,0<u<l-\delta\}} \Lambda(|\nabla_ku(x)|)\psi^\prime(u(x)+\delta)\phi(x) w_k(x)\,dx \notag\\
     & 
     \leq -C_\psi  \int_{\{\nabla_ku \neq 0,0<u<l-\delta\}} \Lambda(|\nabla_ku(x)|) \frac{\psi(u(x)+\delta)}{g(u(x)+\delta)} \phi(x) w_k(x)\,dx \notag \\
     & = -C_\psi I_5,  \label{ineq-3}
\end{align} where 
\begin{align*}
    I_5 = \int_{\{\nabla_ku \neq 0,0<u<l-\delta\}} \Lambda(|\nabla_ku(x)|) \frac{\psi(u(x)+\delta)}{g(u(x)+\delta)} \phi(x) w_k(x)\,dx
\end{align*}
and $C_\psi$ is the constant in the compatibility condition of $(\psi,g)$ defined in Proposition \ref{loc.est}. Applying the monotonicity of the integration on $I_3$, we observe that
\begin{align}
    I_3 & \leq \int_{\{\nabla_ku \neq 0,0<u<l-\delta\}} B(|\nabla_ku(x)|) |\nabla_ku(x)|   |\nabla_k \phi(x)|\psi(u(x)+\delta) w_k(x)\,dx \notag \\
    & = \int_{\{\nabla_ku \neq 0,0<u<l-\delta, \phi \neq0\}}\frac{\Lambda(|\nabla_ku(x)|)}{|\nabla_ku(x)|}   \left(\frac{|\nabla_k \phi(x)|}{\phi(x)}g(u(x)+\delta)\right)\frac{\psi(u(x)+\delta)}{g(u(x)+\delta)}\phi(x) w_k(x)\,dx. \label{I3}
\end{align}
Since $\Lambda$ is an $N$-function, we can apply  Lemma \ref{N-function} to $\Lambda$ with the parameters $s_1= |\nabla_ku(x)|$ and $s_2 = \frac{|\nabla_k \phi(x)|}{\phi(x)}g(u(x)+\delta) $ in \eqref{I3} and yield  
\begin{align}
    I_3 \leq  I_5+ I_6, \label{ineq-4}
\end{align}
where 
\begin{align*}
    I_6  & = \int_{\{\nabla_ku \neq 0,0<u<l-\delta, \phi \neq0\}} \Lambda\left( \frac{|\nabla_k \phi(x)|}{\phi(x)}g(u(x)+\delta) \right) \frac{\psi(u(x)+\delta)}{g(u(x)+\delta)}\phi(x) w_k(x)\,dx.
\end{align*}
Combining the estimates \eqref{ineq-1}, \eqref{ineq-2},  \eqref{ineq-3}, and  \eqref{ineq-4} we get
\begin{align*}
    I  = I_0+ I_1 & \leq  I_2+ I_3+I_4 \leq -C_\psi I_5+ I_5+I_6+I_4.
\end{align*}
Eventually, \begin{align*}
     I_0 + (C_\psi-1) I_5 & \leq I_6+I_4-I_1,
\end{align*}
which gives the required inequality.
\end{proof}

\begin{lemma} \label{Lemma-2-main}
   Under the assumptions of Proposition \ref{loc.est}, letting $\delta \searrow 0$ in \eqref{Lemma-1} yields the following inequality:
    \begin{align}
  &\int_{0<u\leq l} \left( b(x)\Phi(u(x)) + 
    \beta\frac{\Lambda(|\nabla_ku(x)|)}{g(u(x))}
    \right) \psi(u(x)) \phi(x)w_k(x)\,dx \notag  \\
  & \leq   \int_{\{\nabla_ku \neq 0,0<u<l, \phi \neq0\}} \Lambda\left( \frac{|\nabla_k \phi(x)|}{\phi(x)}g(u(x)) \right) \frac{\psi(u(x))}{g(u(x))}\phi(x) w_k(x)\,dx + C_k(l)\label{Lemma2} 
    \end{align}
for $l>0$.  $C_k(l)$ is given in \eqref{C_k}.
\end{lemma}

\begin{proof} Instead of proving the inequality \eqref{Lemma2}, we shall prove 
\begin{align*}
    &\int_{\{\nabla_ku \neq 0,0<u<l, \phi \neq0\}} \Lambda\left( \frac{|\nabla_k \phi(x)|}{\phi(x)}g(u(x)) \right) \frac{\psi(u(x))}{g(u(x))}\phi(x) w_k(x)\,dx  + C_k(l) \\
    & - \int_{0<u\leq l} \left( b(x)\Phi(u(x)) + 
    \beta\frac{\Lambda(|\nabla_ku(x)|)}{g(u(x))}
    \right) \psi(u(x)) \phi(x)w_k(x)\,dx \geq 0.
\end{align*}
For each $\delta\in (0,l)$, in view of \eqref{Lemma-1}, we define the function 
\begin{align*}
    H(\delta, l)(x) = H_1(\delta,l)(x)+ H_2(\delta,l)(x)-H_3(\delta,l)(x),  
\end{align*}
where 
\begin{align*}
    H_1(\delta,l)(x)=&\Lambda\left( \frac{|\nabla_k \phi(x)|}{\phi(x)}g(u(x)+\delta) \right) \frac{\psi(u(x)+\delta)}{g(u(x)+\delta)}\phi(x) \chi_{\{\nabla_ku \neq 0,0<u<l-\delta, \phi \neq0\}} \\
     H_2(\delta,l)(x)=&\psi(l)  \Bigl(
    B(|\nabla_ku(x)|) \langle \nabla_ku(x), \nabla_k\phi(x)\rangle_{\mathbb{R}^n}- b(x) \Phi(u(x))\phi(x)\Bigr) \chi_{\{ u> l-\delta\} }
     \\
      H_3(\delta,l)(x)=& \left( b(x)\Phi(u(x)) + 
    \beta\frac{\Lambda(|\nabla_ku(x)|)}{g(u(x)+\delta)}
    \right) \psi(u(x)+\delta) \phi(x) \chi_{\{0<u\leq l-\delta \}}.
\end{align*}
Let $(\delta_n)$ be a sequence of positive real numbers such that $\delta_n \searrow 0$.  Using the $G$-invariant property of $u$ and Lemma \ref{gradient-msr}, we obtain that the Lebesgue measure of the set
$$ \{ x \in \mathbb{R}^n : \nabla_k u(x) \neq 0,\, u(x) = l \} $$
is zero. This ensures that  $H(\delta_n, l)(x) \to H(0, l)(x) $ a.e. as  $\delta_n \searrow 0$.
The inequality \eqref{Lemma-1} immediately implies that for every $\delta _n$,
$$\int_{\mathbb{R}^n} H(\delta_n, l)(x) \, w_k(x) \, dx \geq 0.$$
Applying the dominated convergence theorem, we obtain
$$\int_{\mathbb{R}^n} H(0, l)(x) \, w_k(x) \, dx \geq 0,$$
which yields the required estimate.
Finally, the proof is completed by showing the existence of an integrable function $ H(l)$  such that $|H(\delta_n, l)(x)| \leq H(l)(x)\, a.e.$\\
We have the following observations: Using the $M$-condition of $\Lambda$, we deduce 
    \begin{align}
        H_1(\delta_n,l)(x) \leq & M_{\Lambda} \Lambda\left( \frac{|\nabla_k \phi(x)|}{\phi(x)} \right) \frac{\Lambda(g(u(x)+\delta_n))}{g(u(x)+\delta_n)}
        {\psi(u(x)+\delta_n)}\phi(x) \chi_{\{\nabla_ku \neq 0,0<u<l-\delta_n, \phi \neq0\}} \notag\\
         \leq & M_{\Lambda} C(\Lambda, g,\psi) \Lambda\left( \frac{|\nabla_k \phi(x)|}{\phi(x)} \right) \phi(x) \chi_{\{\nabla_ku \neq 0,0<u<l, \phi \neq0\}}\notag\\
         =: & H_1(l)(x), \label{H_1}
    \end{align} where $ C(\Lambda, g,\psi)=\sup_{s \in [0,l)} \frac{\Lambda(g(s))}{g(s)}{\psi(s)}$. \\
For $\delta_n < \frac{l}{2}$, the bound for $H_2(\delta_n,l)$ is independent of $\delta_n$: 
\begin{align}
     | H_2(\delta_n,l)(x)|\leq&  \psi(l)  \Bigl(
    B(|\nabla_ku(x)|) |\nabla_ku(x)| |\nabla_k\phi(x)|+ b(x) \Phi(u(x))\phi(x)\Bigr) \chi_{\{ u> \frac{l}{2}\} } \notag\\
     =: &H_2(l)(x). \label{H2}
\end{align}
Recalling the non-increasing nature of $\psi$ and $\zeta$ functions, we obtain the following bound for $H_3(\delta_n,l):$
\begin{align}
    |H_3(\delta_n,l)(x)|  & \leq  \Bigl( b(x)\Phi(u(x))\psi(u(x)) + 
   | \beta|{\Lambda(|\nabla_ku(x)|)}{\zeta(u(x)}
    \Bigl)  \phi(x) \chi_{\{0<u\leq l \}}  \notag\\
    &=: H_3(l)(x)\label{H3}.
\end{align}
In view of \eqref{H_1}, \eqref{H2}, and \eqref{H3}, we have 
\begin{align*}
    |H(\delta_n,l)(x)| \leq H_1(l)(x)+H_2(l)(x)+H_3(l)(x)= H(l)(x). 
\end{align*} 
The integrability of $H(l)$ will follow from the assumptions of the functions.
\end{proof}
\begin{remark}
     The local estimate for the Dunkl–Caccioppoli-type inequality was derived for the class of $G$-invariance functions. For this class of functions, we have the identity $\mathcal{T}_j u(x)= \partial _j u(x)$ for $j=1, 2, \cdots, n$. In the proof of this result, a crucial step is taking the limit $\delta \searrow 0$ in the inequality \eqref{Lemma-1}. The convergence follows from Lemma \ref{gradient-msr}, which ensures that the set \( \{ x : u(x) = l,\, \nabla_k u \neq 0 \} \) has measure zero, where $u$ is $G$-invariant. However, for a general function $f$, such a lemma does not always hold in the Dunkl setting due to the reflection component of the Dunkl gradient. For instance, consider the function
\[
f(x) = 
\begin{cases}
0, & x \leq 0, \\
3x^2 - 2x^3, & 0 < x < 1, \\
1, & x \geq 1.
\end{cases}
\]
Then \( f(x) = 1 \) for \( x \geq 1 \), but the Dunkl operator yields \( \mathcal{T}f(x) = \frac{k}{x} \neq 0 \) if \( k \neq 0 \). So, for a general function, we may not guarantee zero measure for the set \( \{ x\in \mathbb{R}^n: u(x)=l \text{ and } \nabla_ku \neq 0\} \). To resolve this issue, we restrict ourselves to \( G \)-invariant functions \( u \), which eliminate the reflection part and ensure the necessary regularity.
\end{remark}

Now we are stating the global estimate of Dunkl-Caccioppoli-type inequality, which is the limiting case of the local estimate \eqref{Lemma2}.  To the best of our knowledge, the Caccioppoli-type inequalities are not derived for the partial differential inequality $-\Delta_{A}(u) \geq b\Phi(u)\chi_{\{u>0\}}$, which is a special case of $k=0$.
\begin{theorem}\label{thrm-1}(Dunkl Caccioppoli-type estimate):
Let \( u \in W_k^{1,\Lambda}(\mathbb{R}^n) \) be a $G$-invariant,  nonnegative solution to the DDI,
\[
-\Delta_{k,A}u \geq b(x)\Phi(u)\chi_{\{u>0\}}.
\]
 Then $u$ satisfies the Dunkl-Caccioppoli-type inequality
\begin{align*}
\int_{\{u>0\}} &\left( b(x)\Phi(u(x)) + 
\beta\frac{\Lambda(|\nabla_k u(x)|)}{g(u(x))} \right) 
\psi(u(x)) \phi(x) w_k(x)\,dx \\
&\leq  \int_{\{\nabla_k u \neq 0,\ u>0,\ \phi \neq 0\}} \Lambda\left( \frac{|\nabla_k \phi(x)|}{\phi(x)}g(u(x)) \right) \frac{\psi(u(x))}{g(u(x))}\phi(x) w_k(x)\,dx,
\end{align*}
for every nonnegative, compactly supported Lipschitz function \( \phi \) such that
\[
\int_{\{\phi \neq 0,\ \nabla_k u \neq 0\}} \Lambda\left( \frac{|\nabla_k \phi(x)|}{\phi(x)} \right) \phi(x) w_k(x)\,dx < \infty,
\]
where $\beta = C_\psi-1$, and the functions $ \Lambda,u,b,g,\Phi, \psi$, and $\phi$ satisfy the assumptions of Proposition \ref{loc.est}.  
\end{theorem}
\begin{proof}
    The result follows from the monotone convergence theorem by taking the limit $l\to \infty$ in \eqref{Lemma2}.
\end{proof}
The Caccioppoli inequality can be regarded as a reverse form of either the Poincaré inequality or the Sobolev inequality \cite{Giaquinta-2}.  The classical Sobolev inequality estimates the norm of a function in terms of the norm of its gradient, the Caccioppoli inequality works in the opposite direction by bounding the norm of the gradient $\nabla u$ in terms of the norm of the function $u$ itself. In the Dunkl setting, Sobolev-type inequalities have been studied, and readers may refer to \cite{Velichu} for further details.

\section{Nonexistence result related to the Dunkl differential inequality} \label{S:4}
Nonexistence results are central to the study of partial differential equations, inequalities, and their applications in mathematics and modeling. Such results have also been investigated in the Dunkl setting for specific models. In \cite{Jleli-1}, Jleli, Samet, and Vetro established a Liouville-type theorem for the semilinear inequality  $-\Delta_k u \geq |u|^p$ in $\mathbb{R}^n$ and the system of inequalities $-\Delta_k v \geq |u|^p$ and $-\Delta_k u \geq |v|^p$ in $\mathbb{R}^n$ for $n\geq 1$, where $p,q>1$.  Later, in \cite{Jleli-2}, they studied the existence and nonexistence of weak solutions to the semilinear inequality $-\Delta_k u \geq \lambda \frac{u}{|x|^{2}} + |x|^{\alpha} |u|^{p} 
\quad \text{in } B_1 \setminus \{0\}$. Motivated by these results, we establish a sufficient condition for the nonexistence of solutions to the Dunkl differential inequality  $-\Delta_{k, A}u \geq b(x)\Phi(u)\chi_{\{u>0\}},$ in this section. The result is developed through a sequence of supporting lemmas. We begin by stating the main theorem and conclude with its proof, which relies on the Dunkl-Caccioppoli-type inequality.
\\ \vspace{0.15cm}\\
\noindent We need some additional definitions and assumptions for the nonexistence result.
\begin{definition}
    Let $F:[0,\infty) \to \mathbb{R}$ be a function. Then the
Legendre function of $F$ is defined as $$F^*(s)  = \sup_{t>0} \left( st- F(t) \right),$$ where $s \in \text{Dom}(F^*)=\{s\in \mathbb{R}: \sup\limits_{t>0} \left( st- F(t) \right) < \infty \}.$
\end{definition}
\textbf{Assumptions:}
\begin{enumerate}[$(i)$]
    \item $(\Phi_{t,\Lambda,g})$: Let $\Phi:[0,\infty) \to [0,\infty)$ is a continuously differentiable function with $\Phi(0)=0$ and $\Phi(s)>0$ for $s > 0,$ we define \begin{align}
        \Phi_{t,\Lambda,g}(s) =\left( \frac{\Phi(s)^t g(s)}{\Lambda(g(s))}  \right)^{\frac{1}{1-t}}, \quad 0<t<1, \label{phi-funt}
    \end{align} such that, there  exist a constant $C_\Phi >0$ independent of $s$ which satisfies 
    \begin{align*}
        g(s) \Phi_{t,\Lambda,g}^\prime(s)  \leq -C_\Phi  \Phi_{t,\Lambda,g}(s). 
    \end{align*}

    \item $(F_t)$ : Let $ F:[0,\infty)\to [0,\infty)$ be an $N$-function such that $F^*$ is a monotonically increasing $M$-function. There exist constants $D_F$, such that, for every $s>0$, the following inequality holds:
    \begin{align}
    F\left( \Lambda\left( \frac{g(s)}{\Phi_{t,\Lambda,g}(s)} \right) \frac{\Phi_{t,\Lambda,g}(u)}{g(s)} \right) 
    \leq D_F \left( \frac{\Phi(s)}{g(s) B(s)} \right)^{\frac{1}{1-t}}, 
    \label{Prop-F} 
    \end{align}
    where $t\in (0,1)$, and the function $ \frac{F^*(s)}{s}$ is locally bounded.
\end{enumerate}

\begin{theorem} \label{Ex-thrm}
Let $\Lambda$ be a non-decreasing function satisfying the $M$-condition, and define
\begin{align}
    \mathfrak{F}(l)=\frac{\Lambda\left(\frac{1}{l}\right) l^n }{\|b\|_{L^\infty(B(x_0,2l))}^{\frac{t}{1-t}} (F^*\circ \Lambda)^{-1} \left( \frac{\Lambda^{\frac{1}{1-t}} \left( \frac{1}{l}\right)}{F^*\circ \Lambda \left( \frac{1}{l}\right)}\right) }, \text{ for $l>0$ and $0<t<1$}. \label{F-function}
\end{align}

Then the following assertions hold:
\begin{enumerate}[$(i)$]
    \item If $u\in W_{k,\text{loc}}^{1,\Lambda}$ is nonnegative $G$-invariant function and satisfies the DDI, then $\Phi(u)$ is locally integrable and there exists a constant $C>0$ such that 
    \begin{align*}
        \int_{B(x_0,l)\cap \{u>0 \}} b(x)\Phi(u)(x) w_k(x)\,dx \leq C \mathfrak{F}(l),
    \end{align*} for all $x_0\in \mathbb{R}^n$ and $l>0$. 
    \item If $\lim\limits_{l\to\infty} \mathfrak{F}(l)=0$, then there does not exists nonnegative, nonzero solution $u\in W_{k,\text{loc}}^{1,\Lambda}$ satisfying the DDI.
\end{enumerate}    
\end{theorem}

The result in the Euclidean case with the classical gradient was proved using Mitidieri and Pohozaev's \cite{Mitidieri} techniques by identifying a suitable function $\phi$ and an appropriate function $\mathfrak{F}$. Motivated by these techniques, we established the nonexistence result with the aid of a Caccioppoli-type inequality. In Lemmas \ref{Lemma1-4} and \ref{Lemma2-4}, we modify the Caccioppoli-type inequality, after which we construct the function $\mathfrak{F}$ in Lemma \ref{Lemma 2-5-ex} and complete the proof.
\begin{lemma} \label{Lemma1-4}
   Assume that  
\begin{enumerate}[$(i)$]
    \item The function $\Lambda$ is non-decreasing and  satisfies the $M$-condition.
    \item The condition $(\Phi_{t,\Lambda,g})$ holds for $t \in (0,1)$.
\end{enumerate}
 Then there exists a constant $C(\Lambda,t)>0$, such that for every nonnegative $u \in W_{k,\text{loc}}^{1,\Lambda}$ satisfying DDI, we have 
    \begin{align}
        &\int_{\{u>0\}}  \left( b(x)\Phi(u(x)) +2 \beta   \frac{\Lambda(|\nabla_ku(x)|)}{g(u(x))}\right) 
         \Phi_{t,\Lambda,g}(u(x)) \phi(x)w_k(x)\,dx \notag \\
        & \leq C(\Lambda,t) \int_{\{\phi\neq 0, \nabla_k u\neq 0 \}}  b(x)^{\frac{t}{t-1}} \left( \Lambda \left( \frac{|\nabla_k \phi(x)|}{\phi(x)} \right) \right)^{\frac{1}{1-t}} \phi(x) w_k(x)\, dx. \label{Lemma1-4eq}
    \end{align}
\end{lemma}
\begin{proof}
    Using  $\psi=\Phi_{t,\Lambda,g}$ in  \eqref{Lemma2},  we obtain
    \begin{align} 
   & \int_{\{0<u<l\}}\left( b(x)\Phi(u(x)) + 
    \beta\frac{\Lambda(|\nabla_ku(x)|)}{g(u(x))}
    \right) \Phi_{t,\Lambda,g}(u(x)) \phi(x)w_k(x)\,dx \notag \\
    & \leq  \int_{\{\nabla_ku \neq 0,0<u<l, \phi \neq0\}} \Lambda\left( \frac{|\nabla_k \phi(x)|}{\phi(x)}g(u(x)) \right) \frac{\Phi_{t,\Lambda,g}(u(x))}{g(u(x))}\phi(x) w_k(x)\,dx + C_k(l).  \label{Lemma-4-1}
\end{align} Invoking the $M$-conditions of $\Lambda$, we have  
   \begin{align*}
       \Lambda\left( \frac{|\nabla_k \phi(x)|}{\phi(x)}g(u(x)) \right) \leq M_{\Lambda} \Lambda \left(  \frac{|\nabla_k \phi(x)|}{\phi(x)} \right) \Lambda(g(u(x))) .  
   \end{align*} Consequently, the right hand side of \eqref{Lemma-4-1} can be bounded by \begin{align}
      M_{\Lambda} \int_{\{\nabla_ku \neq 0,0<u<l, \phi \neq0\}} \Lambda \left(  \frac{|\nabla_k \phi(x)|}{\phi(x)} \right) \Lambda(g(u(x)))   \frac{\Phi_{t,\Lambda,g}(u(x))}{g(u(x))}\phi(x) w_k(x)\,dx + C_k(l). \label{Lemma-4-3}
   \end{align} 
   Applying Young's inequality $s_1s_2 \leq (1-t)s_1^{\frac{1}{1-t}}+ts_2^{\frac{1}{t}}$ for $0<t<1$,  with the choice  $$s_1=\frac{1}{cb(x)^t}\Lambda \left(  \frac{|\nabla_k \phi(x)|}{\phi(x)} \right) \text{ and } s_2=cb(x)^t \frac{\Lambda(g(u(x)))}{g(u(x))} \Phi_{t,\Lambda,g}(u(x)),$$ where $c>0$ an arbitrary constant and using \eqref{phi-funt}, we observe that \eqref{Lemma-4-3} can be bounded by 
   \begin{align}
        M_{\Lambda}& \frac{(1-t)}{c^\frac{1}{1-t}}
        \int_{\{\nabla_ku \neq 0,0<u<l, \phi \neq0\}}  
        b(x)^{\frac{t}{t-1}} \Lambda \left(\frac{|\nabla_k \phi(x)|}{\phi(x)}  \right)^{\frac{1}{1-t}} \phi(x)w_k(x)\,dx \notag \\
        +& M_{\Lambda} c^{\frac{1}{t}} t  \int_{\{\nabla_ku \neq 0,0<u<l\}}   b(x) \Phi(u(x)) \Phi_{t,\Lambda,g}(u(x)) \phi(x)w_k(x)\,dx + C_k(l). \label{Lemma-4-4}
   \end{align} Taking into account the inequality \eqref{Lemma-4-1} and
   setting $c=\left(\frac{t^{-\frac{1}{t}}}{2M_{\Lambda}}\right)^t$ in  \eqref{Lemma-4-4}, we get
   \begin{align}
      & \int_{\{\nabla_ku \neq 0,0<u<l\}}  \left( b(x) \Phi(u(x)) + 2 \beta \frac{\Lambda(|\nabla_ku(x)|)}{g(u(x))}\right) \Phi_{t,\Lambda,g}(u(x)) \phi(x)w_k(x)\,dx  \notag \\
     \leq &  C(\Lambda,t) 
        \int_{\{\nabla_ku \neq 0,0<u<l, \phi \neq0\}}  
        b(x)^{\frac{t}{t-1}} \Lambda \left(\frac{|\nabla_k \phi(x)|}{\phi(x)}  \right)^{\frac{1}{1-t}} \phi(x)w_k(x)\,dx + 2C_k(l). \notag
 \end{align} Thus, by the monotone convergence theorem, we obtain the result by taking the  $l\to \infty$ on both sides.  
\end{proof}

\noindent Now, we modify Lemma \ref{Lemma1-4} by avoiding the $\nabla_k u$ term in \eqref{Lemma1-4eq} and provide a bound independent of $u$.

\begin{lemma}\label{Lemma2-4}
   Suppose 
\begin{enumerate}[$(i)$]
    \item $\Lambda$ satisfy the $M$-condition.
    \item For $t \in (0,1)$, the conditions $(\Phi_{t,\Lambda,g})$ and $({F}_t)$ hold.
\end{enumerate}
 Then there exists a constant $C(A,\Phi,t, F)>0$ such that, for every nonnegative function $u \in W_{k,\text{loc}}^{1,\Lambda}$ satisfying DDI, and for every nonnegative compactly supported Lipschitz function $\phi$, for which the integrals
    \begin{align*}
        \mathcal{J}_1&:= \int_{\{ \phi \neq0\}}  
        b(x)^{\frac{t}{t-1}} \Lambda \left(\frac{|\nabla_k \phi(x)|}{\phi(x)}  \right)^{\frac{1}{1-t}} \phi(x)w_k(x)\,dx \\ \mathcal{J}_2&:=  \int_{\{\phi \neq 0 \}} b(x)
        F^*\left( \frac{1}{b(x)} \right)
        F^*\left( \Lambda\left( \frac{|\nabla_k\phi(x)|}{\phi(x)}\right)  \right) \phi(x)w_k(x) 
    \end{align*} are finite, we have 
    \begin{align*}
        \int_{\{ u>0 \}} b(x) \Phi(u(x))\phi(x)w_k(x)\,dx \leq C(A,\Phi,t, F) \mathcal{G}\left( \frac{\mathcal{J}_1}{\mathcal{J}_2} \right) \mathcal{J}_2,
    \end{align*} where $\mathcal{G}(s)= \frac{s}{\left(F^*\circ\Lambda\right)^{-1}(s)}$.
\end{lemma}
\begin{proof}
Using the relation \eqref{W.S of Laplace}, we have 
\begin{align*}
    \int_{\{u>0\}}b(x) \Phi(u(x))\phi(x)w_k(x)\,dx   & \leq 
    \int_{\{\nabla_ku\neq 0, u>0\}} B(|\nabla_ku(x)| |\nabla_ku(x)| |\nabla_k\phi(x)|w_k(x)\,dx 
    \\
    &  = \epsilon \int_{\{\nabla_ku\neq 0, u>0\}}  \frac{\Lambda(|\nabla_k u(x)|) }{|\nabla_ku(x)|}  \left( \frac{|\nabla_k\phi(x)|}{ \epsilon \phi(x)} \,\frac{g(u(x))}{\Phi_{t,\Lambda,g}(u(x))}
     \right) \\
     &\quad \times \frac{\Phi_{t,\Lambda,g}(u(x))}{g(u(x))}\phi(x)w_k(x)\,dx,
\end{align*} where $\epsilon>0$, is an arbitrary constant. Applying Lemma \ref{N-function} to the function $\Lambda$ with $$s_1=|\nabla_k u(x)|  \text{ and } s_2=\frac{|\nabla_k\phi(x)|}{\epsilon \phi(x)} \,\frac{g(u(x))}{\Phi_{t, \Lambda,g}(u(x))}$$ the integral boils down to
\begin{align}
     \int_{\{u>0\}}b(x) \Phi(u(x))\phi(x)w_k(x)\,dx &\leq  \epsilon\left(\mathcal{J}_3 + \mathcal{J}_4\right), \label{Lemma-4-2-J1}
\end{align}
where 
\begin{align*}
    \mathcal{J}_3 &= \int_{\{\nabla_ku\neq 0, u>0\}} \Lambda(|\nabla_k u(x)| \frac{\Phi_{t,\Lambda,g}(u(x))}{g(u(x))}\phi(x)w_k(x)\,dx \\
    \mathcal{J}_4 &=\int_{\{ \nabla_k\neq 0, \phi \neq 0\}} \Lambda \left( \frac{1}{\epsilon} \frac{|\nabla_k\phi(x)|}{\phi(x)} \,\frac{g(u(x))}{\Phi_{t,\Lambda,g}(u(x))}
     \right) \frac{\Phi_{t,\Lambda,g}(u(x))}{g(u(x))}\phi(x)w_k(x)\,dx.
\end{align*}
Invoking Lemma \ref{Lemma1-4}, we obtain 
\begin{align}
    \mathcal{J}_3 \leq C(\Lambda,t) \int_{\{\phi\neq 0, \nabla_k u\neq 0 \}}  b(x)^{\frac{t}{t-1}} \left( \Lambda \left( \frac{|\nabla_k \phi(x)|}{\phi(x)} \right) \right)^{\frac{1}{1-t}} \phi(x) w_k(x) dx \leq C(\Lambda,t)\mathcal{J}_1. \label{Lemma-4-2-J2}
\end{align}
The $M$-condition of $\Lambda$ yields the following estimate 
\begin{align}
    \mathcal{J}_4 \leq & M_{\Lambda}^2 \Lambda\left( \frac{1}{\epsilon} \right) \int_{\{ \nabla_ku\neq 0, \phi\neq 0\}} \Lambda\left( \frac{|\nabla_k \phi(x)|}{\phi(x)} \right) \Lambda\left( \frac{g(u(x))}{\Phi_{t,\Lambda,g}(u(x))} \right) \frac{\Phi_{t,\Lambda,g}(u(x))}{g(u(x))}\phi(x)w_k(x)\,dx \notag \\
    =& M_{\Lambda}^2 \Lambda\left( \frac{1}{\epsilon} \right)\delta \int_{\{ \nabla_ku\neq 0, \phi\neq 0\}} Z_1(x) \frac{Z_2(x)}{\delta b(x)}b(x) \phi(x)w_k(x)\,dx, \label{Lemma4-2-1}
\end{align}
where $$Z_1(x)= \Lambda\left( \frac{g(u(x))}{\Phi_{t,\Lambda,g}(u(x))} \right) \frac{\Phi_{t,\Lambda,g}(u(x))}{g(u(x))} \text{  ,  }  Z_2(x)= \Lambda\left( \frac{|\nabla_k \phi(x)|}{\phi(x)} \right),$$ and $\delta$ is an arbitrary positive real number. Applying Fenchel-Young inequality $(st\leq F(s)+F^*(t))$ to the function $F$ with the parameters $s= Z_1(x)$ and $t=\frac{Z_2(x)}{\delta b(x)}$, along with the inequality \eqref{Prop-F}, we obtain, 

\begin{align}
     Z_1(x) \frac{Z_2(x)}{\delta b(x)} & \leq  F\left( \Lambda\left( \frac{g(u(x))}{\Phi_{t,\Lambda,g}(u(x))} \right) \frac{\Phi_{t,\Lambda,g}(u(x))}{g(u(x))} \right) + F^* \left( \frac{1}{\delta b(x)} \Lambda\left( \frac{|\nabla_k \phi(x)|}{\phi(x)}\right)  \right) \notag \\
     & \leq D_F \left( \frac{\Phi(u(x))}{g(u(x))B(u(x))} \right)^{\frac{1}{1-t}} + F^* \left( \frac{1}{\delta b(x)} \Lambda\left( \frac{|\nabla_k \phi(x)|}{\phi(x)}\right) \right). \label{Lemma-4-2-2}
\end{align}
In light of \eqref{phi-funt}, \eqref{Lemma-4-2-2},  and Lemma \ref{Lemma-4-1}, the bound \eqref{Lemma4-2-1} can be modified as 
\begin{align}
 \mathcal{J}_4 & \leq    M_{\Lambda}^2 \Lambda\left( \frac{1}{\epsilon} \right) D_F\delta \Bigg\{ \int_{\{ \nabla_ku\neq 0, \phi\neq 0\}}  \left( \frac{\Phi(u(x))}{g(u(x))B(u(x))} \right)^{\frac{1}{1-t}} b(x)\phi(x)w_k(x)\,dx \notag \\
  & \hspace{2cm}+ 
    \int_{\{ \nabla_ku\neq 0, \phi\neq 0\}}
 F^* \left( \frac{1}{\delta b(x)} \Lambda\left( \frac{|\nabla_k \phi(x)|}{\phi(x)}\right)  \right) b(x)\phi(x)w_k(x)\,dx \Bigg\} \notag \\
 & \leq  M_{\Lambda}^2 \Lambda\left( \frac{1}{\epsilon} \right) D_F\delta \Bigg\{ C(\Lambda,t) 
        \int_{\{ \phi \neq0\}}  
        b(x)^{\frac{t}{t-1}} \Lambda \left(\frac{|\nabla_k \phi(x)|}{\phi(x)}  \right)^{\frac{1}{1-t}} \phi(x)w_k(x)\,dx \notag \\
    & \hspace{2cm}+ 
    \int_{\{ \nabla_ku\neq 0, \phi\neq 0\}}
 F^* \left( \frac{1}{\delta b(x)} \Lambda\left( \frac{|\nabla_k \phi(x)|}{\phi(x)}\right)  \right) b(x)\phi(x)w_k(x)\,dx \Bigg\} \notag \\
 & \leq  M_{\Lambda}^2 \Lambda\left( \frac{1}{\epsilon} \right) D_F\delta \mathcal{J}, \label{Lemma-4-2-J3}
\end{align} where $$ \mathcal{J}= C(\Lambda,t) \mathcal{J}_1 + M_{F^*}^2 F^*\left(\frac{1}{\delta}\right)\mathcal{J}_2 . $$
Combining the inequalities \eqref{Lemma-4-2-J1}, \eqref{Lemma-4-2-J2}, and \eqref{Lemma-4-2-J3}, we have 
\begin{align}
     \int_{\{u>0\}}b(x) \Phi(u(x))\phi(x)w_k(x)\,dx  \leq \epsilon C(\Lambda,t)\mathcal{J}_1 + \epsilon M_{\Lambda}^2 \Lambda\left( \frac{1}{\epsilon} \right) D_F\delta \mathcal{J}. \label{Lemma-4-2-J4}
\end{align}
If $\mathcal{J}=0$, the result follows directly.  
Now, assume $\mathcal{J}\neq 0$, we choose 
$ \frac{1}{\epsilon} = \Lambda^{-1} \left(\frac{C(\Lambda,t)\mathcal{J}_1 }{M_{\Lambda}^2 D_F\delta \mathcal{J}}
\right)$ then \eqref{Lemma-4-2-J4} becomes 
\begin{align}
     \int_{\{u>0\}}b(x) \Phi(u(x))\phi(x)w_k(x)\,dx  \leq \frac{2C(\Lambda,t)\mathcal{J}_1 }{ \Lambda^{-1} \left(\frac{C(\Lambda,t)\mathcal{J}_1 }{M_{\Lambda}^2 D_F\delta \mathcal{J}}
\right)} \label{Lemma-4-2-J}
\end{align} for every $\delta>0$. For particular value of $\delta$, $\frac{1}{\delta}= (F^{*})^{-1}\left( \frac{C(\Lambda,t) \mathcal{J}_1}{M_{F^*}^2 \mathcal{J}_2 } \right),$ the inequality  \eqref{Lemma-4-2-J} boils down to the form 
\begin{align*}
      \int_{\{u>0\}}b(x) \Phi(u(x))\phi(x)w_k(x)\,dx  \leq  \frac{2C(\Lambda,t)\frac{\mathcal{J}_1 }{\mathcal{J}_2}}{ \Lambda^{-1} \left( \frac{1}{2M_{\Lambda}^2 D_F\delta} (F^*)^{-1} \left( \frac{C(\Lambda,t)}{M_{F^*}^2} \frac{\mathcal{J}_1}{\mathcal{J}_2}  \right)
\right)
      } \mathcal{J}_2.
\end{align*}
Thus, the desired result will follow from the properties of $(F^*)^{-1}$ and $\Lambda^{-1}$. 
\end{proof}
We use the following Lipschitz function $\eta$ on $[0,\infty)$ as the auxiliary function
\[
\eta(s) = \begin{cases}
          1 & \text {if \,\,}  0\leq s \leq 1, \\
          -s+2  & \text {if \,\,} 1\leq s \leq 2, \\
          0  & \text {if \,\,} s\geq 2.
       \end{cases}
    \]
Let $r\geq 1$, we define the compactly supported Lipschitz function $\eta_{l,x_0}^r$ on $\mathbb{R}^n$ by   $\eta_{l,x_0}^r(x) =\left( \eta\left( \frac{|x_0-x|}{l}\right)\right) ^r$, and set 
\begin{align*}
    \mathcal{J}_1^r(l) = & \int_{B(x_0,2l)} b(x)^{\frac{t}{t-1}}  \left(\Lambda \left(  \frac{|\nabla_k \eta_{l,x_0}^r(x)|} {\eta_{l,x_0}^r(x)}\right)  \right)^{\frac{1}{1-t}} \eta_{l,r}(x)w_k(x)\,dx,\\
    \mathcal{J}_2^r(l) = & \int_{B(x_0,2l)} b(x) F^*\left( \frac{1}{b(x)} \right) F^*\left(\Lambda  \left(  \frac{|\nabla_k \eta_{l,x_0}^r(x)|} {\eta_{l,x_0}^r(x)}\right)\right)\eta_{l,x_0}^r(x)w_k(x)\,dx.
    \end{align*}
The following lemma details some of the crucial properties of the above defined functions $\mathcal{J}_1^r(l), \mathcal{J}_2^r(l)$ and $\eta_{l,x_0}^r$.

\begin{lemma} \label{Lemma-2 in ex}
    If $U$ is a non-decreasing $N$-function satisfying the $M$-condition, then there exists an $r\geq0$ such that 
    \begin{enumerate}[$(i)$]
        \item $\int_0^1 U\left( \frac{1}{s}\right)s^r ds < \infty$
        \item 
            $\int_{B(x_0,2l)} b(x)^{\frac{t}{t-1}} U \left(  \frac{|\nabla_k \eta_{l,x_0}^r(x)|} {\eta_{l,x_0}^r(x)}\right)  \eta_{l,r}(x)w_k(x)\,dx \leq 
             \|b\|_{L^\infty(B(x_0,2l))}^{\frac{t}{t-1}} C(r,l)$,
        \item 
           $ \int_{B(x_0,2l)}  b(x) F^*\left( \frac{1}{b(x)}\right)  U  \left(  \frac{|\nabla_k \eta_{l,x_0}^r(x)|} {\eta_{l,x_0}^r(x)}\right) \eta_{l,r}(x)w_k(x)\,dx 
             \leq C(F^*,b)C(r,l) $,
             
    \end{enumerate}
         where 
         \begin{align*}
             C(r,l) =& U\left( \frac{r}{l}\right)M_U U(C(\eta_{l,x_0}^r,k)+1) \\
             & \times \left( w_k(B(x_0, l))+ M_U \|w_k\|_{L^\infty(B(x_0,2l))}\theta_{n-1}2^{n-1}l^n \int_0^1 U\left( \frac{1}{s}\right)s^r ds    
        \right).
         \end{align*}
     
\end{lemma}

\begin{proof}
    We begin the proof by recalling the $M$-condition of $U$. It immediately follows that there exists some $\tau>0$ such that $U(s)\leq cs^\tau U(1)$, for all $s>1$. Then the integral $ \int_0^1 U\left( \frac{1}{s}\right) s^r ds = \int_1^\infty U(s)\frac{1}{s^{2+r}}ds$ converges for $r>\tau-1$. As a result, we choose $r > \max\{1, \tau-1\}$. Now consider the Dunkl gradient 
    \begin{align*}
        \nabla_k \eta_{l,x_0}^r & = \left( \mathcal{T}_1 \eta_{l,x_0}^r, \mathcal{T}_2 \eta_{l,x_0}^r,\cdots , \mathcal{T}_n \eta_{l,x_0}^r \right)\\ 
        & = \left(\partial_1 \eta_{l,x_0}^r, \partial_2 \eta_{l,x_0}^r, \cdots , \partial_n \eta_{l,x_0}^r \right)+ \left( \mathfrak{R}_1(\eta_{l,x_0}^r), \mathfrak{R}_2(\eta_{l,x_0}^r), \cdots , \mathfrak{R}_n(\eta_{l,x_0}^r)  \right)\\
        & = \nabla \eta_{l,x_0}^r + \mathfrak{R}(\eta_{l,x_0}^r),
    \end{align*} 
    where  \begin{align*}
        \mathfrak{R}_j(\eta_{l,x_0}^r)(x) = \sum _{\alpha \in \mathcal{R}_+}k(\alpha) \alpha(j) \frac{\eta_{l,x_0}^r(x)-\eta_{l,x_0}^r(\sigma_\alpha(x))}{\langle \alpha, x \rangle}, \quad \text{ for } j=1,2,\cdots, n.
    \end{align*}
   
We observe that 
\begin{align}
    |\nabla \eta_{l,x_0}^r(x)| \leq \frac{r}{l} \left(-\frac{|x-x_0|}{l}+2 \right)^{r-1}\chi_{\{ l\leq |x-x_0| \leq 2l\}} \quad \text{a.e.,} \label{gradient}
\end{align} and using the Lipschitz property of $\eta_{l,x_0}^r$, for each $x\in B(x_0,2l)$ we obtain 
\begin{align}
    |\mathfrak{R}_j(\eta_{l,x_0}^r)(x)| & \leq \sqrt{2} \sum _{\alpha \in \mathcal{R}_+}k(\alpha) \bigg| \frac{\eta_{l,x_0}^r(x)-\eta_{l,x_0}^r(\sigma_\alpha(x))}{\langle \alpha, x \rangle} \bigg| \notag \\
    & \leq C(\eta_{l,x_0}^r,k)\frac{r}{l} \Bigg\{ \left(-\frac{|x-x_0|}{l}+2 \right)^{r-1}\chi_{\{ l\leq |x-x_0| < 2l\}} + \chi_{B(x_0,l)}\Bigg\} \label{reflction}
\end{align} a.e. for some generic constant $C(\eta_{l,x_0}^r,k)>0$.
Using \eqref{gradient} and \eqref{reflction} we compute the bound for the Dunkl gradient 
\begin{align*}
     |\nabla_k \eta_{l,x_0}^r(x)| \leq \frac{r}{l} \left( \left(-\frac{|x-x_0|}{l}+2 \right)^{r-1}\chi_{\{ l\leq |x-x_0| \leq 2l\}} \left( 1+ C(\eta_{l,x_0}^r,k) \right) + C(\eta_{l,x_0}^r,k)\chi_{B(x_0,l)}\right).
\end{align*}
Invoking the properties of $U$ we have  
\begin{align*}
    & \int_{B(x_0,2l)} b(x)^{\frac{t}{t-1}}  U \left(  \frac{|\nabla_k \eta_{l,x_0}^r(x)|} {\eta_{l,x_0}^r(x)}\right)   \eta_{l,r}(x)w_k(x)\,dx\\
    & \leq U\left( \frac{rC(\eta_{l,x_0}^r,k)}{l} \right)
    \int_{B(x_0,l)} b(x)^{\frac{t}{t-1}}    w_k(x)\,dx + \int_{\{l< |x-x_0|<2l \}} b(x)^{\frac{t}{t-1}}  U \left(  \frac{|\nabla_k \eta_{l,x_0}^r(x)|} {\eta_{l,x_0}^r(x)}\right)   \eta_{l,r}(x)w_k(x)\,dx\\ 
    & \leq \|b\|_{L^\infty(B(x_0,2l))}^{\frac{t}{t-1}}U\left( \frac{r}{l}\right)M_U  \Bigg\{ U(C(\eta_{l,x_0}^r,k)) w_k(B(x_0, l)) \\
    & \quad + M_U U(C(\eta_{l,x_0}^r,k)+1)\int_{\{l\leq |x-x_0|\leq 2l\}} U \left( \frac{ 1 }{-\frac{|x-x_0|}{l}+2 } \right) \left(-\frac{|x-x_0|}{l}+2 \right)^{r} w_k(x)\,dx
    \Bigg\}\\
    & \leq C_1 + C_2 \int_{0}^1 U\left( \frac{1}{s} \right) s^r ds,
\end{align*} where $C_1= \|b\|_{L^\infty(B(x_0,2l))}^{\frac{t}{t-1}}U\left( \frac{r}{l}\right)M_U  U(C(\eta_{l,x_0}^r,k)) w_k(B(x_0, l))$ and $$\text{ \,\,\, } C_2= \|b\|_{L^\infty(B(x_0,2l))}^{\frac{t}{t-1}}U\left( \frac{r}{l}\right)M_U^2U(C(\eta_{l,x_0}^r,k)+1)\|w_k\|_{L^\infty(B(x_0,2l))}\theta_{n-1}2^{n-1}l^n . $$
Here $\theta_{n-1}$ is the surface measure of $\mathbb{S}^{n-1}$ with respect to the restricted Lebesgue measure on $\mathbb{R}^n$. Thus, we have the desired result. Property $(iii)$ will directly follow from the fact that $ \sup_{x\in \overline{B(x_0, 2l)}}\frac{F^*\left(\frac{1}{b(x)}\right)}{\frac{1}{b(x)}}$ is finite and property $(ii)$.

\end{proof}
The following lemma is crucial in the proof of nonexistence.
\begin{lemma}\cite{Kalamajska-1} \label{Lemma-3 in ex}
    Let $U$ be an arbitrary strictly convex function. Then for any $s_1,s_2>0$, the functions $\Psi_1(s_1)= \frac{{s_1}}{U^{-1}\left(\frac{s_1}{s_2}\right)}$ and $\Psi_2(s_2)= \frac{{s_1}}{U^{-1}\left(\frac{s_1}{s_2}\right)}$  are increasing on $[0,\infty).$
\end{lemma}

\begin{lemma} \label{Lemma 2-5-ex}
    Let $\mathcal{O}_r(l)=\mathcal{G}\left( \frac{\mathcal{J}_1^r(l)}{\mathcal{J}_2^r(l)} \right) \mathcal{J}_2^r(l) $ where $\mathcal{G}(s)$ is given in Lemma \ref{Lemma2-4}. Then there exists an $r>0$ such that 
    $$\mathcal{O}_r(l) \leq  C(r,\eta_{l,x_0}^r, \Lambda, F^*)  \mathfrak{F}(l),$$ 
    where $\mathfrak{F}$ is given in \eqref{F-function} and the constant $ C(r,\eta_{l,x_0}^r, \Lambda, F^*) $ is independent of $l$.
\end{lemma}
\begin{proof}
   We begin the proof by considering $U_1= \Lambda^\frac{1}{1-t}$  and $U_2= F^*\circ\Lambda$ in Lemma \ref{Lemma-2 in ex} and also choose an $r>0$ such that the functions $\mathcal{J}_1^r(l)$ and $\mathcal{J}_2^r(l)$ are finite  for each $l>0$. In particular, we have 
   \begin{align*}
       \mathcal{J}_1^r(l) & \leq   \left( \Lambda\left( \frac{r}{l}\right)\Lambda(C(\eta_{l,x_0}^r,k)+1) \right)^{\frac{1}{1-t}}C_1(r,l) \\
       \mathcal{J}_2^r(l) & \leq  F^*\left( \Lambda\left( \frac{r}{l}\right)\right)F^*\bigg(\Lambda(C(\eta_{l,x_0}^r,k)+1)\bigg) C_2(r,l).
   \end{align*} Note that, for a large value of $l$, we obtain 
   \begin{align*}
       C_1(r,l) = &M_{\Lambda^{\frac{1}{1-t}}}  \|b\|_{L^\infty(B(x_0,2l))}^{\frac{t}{t-1}}  \\
       & \times \left( w_k(B(x_0, l))+ M_{\Lambda^{\frac{1}{1-t}}} \|w_k\|_{L^\infty(B(x_0,2l))}\theta_{n-1}2^{n-1}l^n \int_0^1 \Lambda^{\frac{1}{1-t}}\left( \frac{1}{s}\right)s^r ds    
        \right)\\
        & \leq C_1({\Lambda}) \|b\|_{L^\infty(B(x_0,2l))}^{\frac{t}{t-1}} \theta_{n-1}2^{n}l^n \int_0^1 \Lambda^{\frac{1}{1-t}}\left( \frac{1}{s}\right)s^r ds:= A_1,     
   \end{align*} where $ C_1({\Lambda})=M_{\Lambda^{\frac{1}{1-t}}} \max\{ w_k(B(x_0, l)), M_{\Lambda^{\frac{1}{1-t}}} \|w_k\|_{L^\infty(B(x_0,2l))}\}$. In a similar way, we can derive for $ C_2(r,l)$ as
   \begin{align*}
       C_2(r,l) = &M_{F^*\circ \Lambda} C(F^*,b)  \\
       & \times \left( w_k(B(x_0, l))+ M_{F^*\circ \Lambda} \|w_k\|_{L^\infty(B(x_0,2l))}\theta_{n-1}2^{n-1}l^n \int_0^1 F^*\circ \Lambda\left( \frac{1}{s}\right)s^r ds    
        \right)\\
        & \leq C_2({F^*\circ \Lambda})\theta_{n-1}2^{n-1}l^n \int_0^1 F^*\circ \Lambda\left( \frac{1}{s}\right)s^r ds :=A_2,
   \end{align*} where $C_2({F^*\circ \Lambda})=M_{F^*\circ \Lambda}  \max\{ w_k(B(x_0, l)), M_{F^*\circ \Lambda} \|w_k\|_{L^\infty(B(x_0,2l))}\} $.
   Using Lemma \ref{Lemma-3 in ex}, we obtain 
   \begin{align*}
       \mathcal{G}\left( \frac{\mathcal{J}_1^r(l)}{\mathcal{J}_2^r(l)} \right) \mathcal{J}_2^r(l) & = \frac{\frac{\mathcal{J}_1^r(l)}{\mathcal{J}_2^r(l)}}{(F^*\circ \Lambda)^{-1}\left(\frac{\mathcal{J}_1^r(l)}{\mathcal{J}_2^r(l)} \right)} \mathcal{J}_2^r(l)\\
       & \leq \frac{\left( \Lambda\left( \frac{r}{l}\right)\Lambda(C(\eta_{l,x_0}^r,k)+1) \right)^{\frac{1}{1-t}} A_1 }{(F^*\circ \Lambda)^{-1}\left(\frac{\left( \Lambda\left( \frac{r}{l}\right)\Lambda(C(\eta_{l,x_0}^r,k)+1) \right)^{\frac{1}{1-t}} A_1 }{F^*\left( \Lambda\left( \frac{r}{l}\right)\right)F^*\bigg(\Lambda(C(\eta_{l,x_0}^r,k)+1)\bigg)A_2 }  \right)}\\
& \leq C(r,\eta_{l,x_0}^r, \Lambda, F^*) \frac{\Lambda\left(\frac{1}{l}\right) l^n }{\|b\|_{L^\infty(B(x_0,2l))}^{\frac{t}{1-t}} (F^*\circ \Lambda)^{-1} \left( \frac{\Lambda^{\frac{1}{1-t}}\left( \frac{1}{l}\right)}{F^*\circ \Lambda\left( \frac{1}{l}\right)} \right)}.      
   \end{align*}
   Here, the last inequality follows from the properties of the function $\Lambda$ and $F^*$, as well as their inverses.
\end{proof}

\begin{proof} of Theorem \ref{Ex-thrm}: Let $x_0\in \mathbb{R}^n$ be an arbitrary vector. Choosing $\phi= \eta_{l,x_0}^r$ in Lemma \ref{Lemma2-4} with 
 some $r\geq 1$ we notice that
\begin{align*}
     \int_{B(x_0,l)\cap \{u>0 \}}  b(x)\Phi(u)(x) w_k(x)\,dx & \leq \int_{\{ u>0 \}} b(x) \Phi(u(x))\eta_{l,x_0}^r(x)w_k(x)\,dx\\ 
     & \leq C(A,\Phi,t, F) \mathcal{O}_r(l)\\
     & \leq C(A,\Phi,t, F)  C(r,\eta_{l,x_0}^r, \Lambda, F^*) \mathfrak{F}(l).
\end{align*}
The property $(ii)$ follow from the fact that $\lim\limits_{l\to\infty}\mathfrak{F}(l)=0$. Consequently, the integral $\int_{\{u>0\}}b(x)\Phi(u(x))w_k(x)\,dx=0$, and thus we conclude that there is only a trivial solution.
\end{proof}

 \noindent\textbf{Acknowledgment:}
 The authors express their gratitude for the financial support provided by the UGC (221610147795) for the first author (AP) and the Anusandhan National Research Foundation (ANRF) with the reference number SUR/2022/005678, for the second author (SKV), and their funding and resources were instrumental in the completion of this work. \\  
\\
 \textbf{Declarations:}\\
 \textbf{Conflict of interest:}
 We would like to declare that we do not have any conflict of interest.\\ \\
 \textbf{Data availability:}
 No data sets were generated or analyzed during the current study.
\bibliographystyle{abbrvnat} 

\begin{thebibliography}{Md.}



\bibitem{Badiale}M. Badiale, G. Tarantello, A Sobolev-Hardy inequality with applications to a nonlinear elliptic equation arising in astrophysics. Arch. Rational Mech. Anal. \textbf{163}:259–293 (2002).

\bibitem{Baldelli}L. Baldelli, R. Filippucci, A priori estimates for elliptic problems via Liouville type theorems. Discrete Contin. Dyn. Syst. Ser. S. \textbf{13}:1883-1898 (2020).

\bibitem{Bidaut}M. F. Bidaut V\'eron, S. Pohozaev, Nonexistence results and estimates for some nonlinear elliptic problems. J. Anal. Math. \textbf{84}(1):1-49 (2001).

 \bibitem{Caccioppoli}R. Caccioppoli,Limitazioni integrali per le soluzioni di un'equazione lineare ellitica a derivate parziali. Giorn. Mat. Battaglini. \textbf{4}(80):186–212 (1951).

\bibitem{Caristi} G. Caristi, L. D’Ambrosio,  E. Mitidieri, Liouville theorems for some nonlinear inequalities. Proc. Steklov Inst. Math. \textbf{260}(1):90-111 (2008).

 \bibitem{Chlebicka}I. Chlebicka, P. Drábek, A. Ka\l amajska,   Caccioppoli-type estimates and Hardy-type inequalities derived from weighted $p$-harmonic problems. Rev. Mat. Complut. \textbf{32}:(3) 601-630 (2019).

 \bibitem{DJ1} M.F.E. de Jeu, The Dunkl transform. Invent. Math. \textbf{113}:147-162 (1993).

 \bibitem{DV} V.J.F. Diejen, L. Vinet,  Calogero-Sutherland-Moser models. CRM Series in Mathematical Physics. Springer-Verlag, (2000).

\bibitem{Ding}S. Ding,  Two-weight Caccioppoli inequalities for solutions of nonhomogeneous A-harmonic equations on Riemannian manifolds. Proc. Am. Math. Soc. \textbf{132}(8):2367-2375 (2004).

\bibitem{Donaldson}T.K. Donaldson,  N.S. Trudinger,   Orlicz-Sobolev spaces and imbedding theorems. J. Funct. Anal. \textbf{8}(1):52-75 (1971).


 \bibitem{D1} C.F. Dunkl, Differential-difference operators associated to reflection groups. Trans. Amer. Math. \textbf{311}(1):167-183 (1989).

 \bibitem{D2}C.F. Dunkl, Integral kernels with reflection group invariance. Canad. J. Math. \textbf{43}(6):1213-1227 (1991).

 \bibitem{D3} C.F. Dunkl, Hankel transforms associated to finite reflection groups. In: Proc. of the special session on hypergeometric functions on domains of positivity, Jack polynomials, and applications. Proceedings, Tampa 1991, Contemp. Math. 123-138 (1992).

 \bibitem{Esteban}M.J. Esteban,  P. L. Lions, Existence and non-existence results for semilinear elliptic problems in unbounded domains. Proc. Roy. Soc. Edinburgh Sect. A, \textbf{93}(1-2):1-14 (1982).

 \bibitem{Gallardo}L. Gallardo, L. Godefroy,  Propri\'e\'t e de Liouville et \'equation de Poisson pour le  Laplacien g\'en\' erali\'e de Dunkl. C.R.Math.Acad.Sci.Paris \textbf{337}:639–644 (2003).

\bibitem{Gariepy}R.F. Gariepy,  A Caccioppoli inequality and partial regularity in the calculus of variations. Proc. Roy. Soc. Edinburgh Sect. A. \textbf{112}(3-4):249-255 (1989). 

 \bibitem{Giaquinta}M. Giaquinta, Multiple integrals in the calculus of variations and nonlinear elliptic systems. Princeton University Press, Princeton (1983).

 \bibitem{Giaquinta-1}M. Giaquinta,  J.  Sou\v cek, Caccioppoli's inequality and Legendre-Hadamard condition. Math. Ann. \textbf{270}(1):105-107 (1985).

\bibitem{Giaquinta-2}M. Giaquinta, L. Martinazzi,   $L^2$-regularity: The Caccioppoli inequality. In: An Introduction to the Regularity Theory for Elliptic Systems, Harmonic Maps and Minimal Graphs. Publications of the Scuola Normale Superiore. Edizioni della Normale, Pisa (2012). 

\bibitem{Gossez}J.P. Gossez, Nonlinear elliptic boundary value problems for equations with rapidly
 (or slowly) increasing coefficients. Trans. Amer. Math. Soc. \textbf{190}:163-205 (1974).

 \bibitem{Guliyev} V. Guliyev, Y. Mammadov,   F. Muslumova, Characterization of fractional maximal operator and its commutators on Orlicz spaces in the Dunkl setting. J. Pseudo-Differ. Oper. Appl. \textbf{11}:1699–1717 (2020).

\bibitem{Harjulehto}P. Harjulehto, P. H\"ast\"o, Orlicz spaces and generalized {O}rlicz spaces, Lecture Notes in Mathematics. 2236- Springer, Cham (2019).

 \bibitem{Hu}J.E.S.  Humphreys, Reflection groups and Coxeter groups. Cambridge University Press, Cambridge, (1990).

 \bibitem{Jleli-3}M. Jleli, B. Samet,  Nonlinear Liouville-type theorem for a differential inequality involving Dunkl-Baouendi-Grushin operator. Discrete Contin. Dyn. Syst. Ser. S \textbf{18}(2):426-443 (2025).

 \bibitem{Jleli-1}M. Jleli, B. Samet,  C. Vetro,   Liouville-type results for semilinear inequalities involving the Dunkl Laplacian operator. Complex Var. Elliptic Equ. \textbf{70}(8):1321–1336 (2024).

 \bibitem{Jleli-2}M. Jleli, B. Samet,  C. Vetro,  On semilinear inequalities involving the Dunkl Laplacian and an inverse-square potential outside a ball. Adv. Nonlinear Anal. \textbf{13}(1): 20240046 (2024).

\bibitem{Jleli}M. Jleli, B. Samet, S. Zeng,   On a semilinear inequality with Dunkl Laplacian and inverse-square potential on a ball. Results Math. \textbf{80}(1):10 (2025).

 \bibitem{Kalamajska-1} A. Ka\l amajska,  K. Pietruska-Pa\l uba, I. Skrzypczak, Nonexistence results for differential inequalities involving {$A$}-Laplacian. Adv. Differ. Equ. \textbf{17}(3-4):307-336 (2012). 

 \bibitem{RKAN} R. Kane, Reflection groups and invariant theory. CMS Books Math. 5, Springer–Verlag, New York, (2001).

\bibitem{Matukuma}T. Matukuma, The Cosmos, Iwanami Shoten, Tokyo, 1938.

\bibitem{Mitidieri}E. Mitidieri, S. Pohozaev, Nonexistence of positive solutions to quasilinear elliptic problems in $\mathbb{R}^n$, Proc. Steklov. Inst. Math., 227 (1999), 186216, (translated from Tr. Mat. Inst. Steklova, 227 (1999), 192222).

\bibitem{Ni} W.M. Ni, J. Serrin, Nonexistence theorems for quasilinear partial differential
 equations, Proceedings of the conference commemorating the 1st centennial of the
 Circolo Matematico di Palermo (Palermo, 1984). Rend. Circ. Mat. Palermo (2) Suppl.
 No. 8, 171185 (1985).

\bibitem{Niu}J. Niu, G. Shi, Y. Xing,   Higher order {P}oincar\'e{} inequality and {C}accioppoli inequality with {O}rlicz norms for differential forms. J. Math. Inequal. \textbf{18}(4):1217-1232 (2024).

 \bibitem{Opdam} E.M. Opdam, Dunkl operators, Bessel functions and the discriminant of a finite Coxeter group, Compositio Math. \textbf{85}(3):333–373 (1993).

\bibitem{Ren}G. Ren, L. Liu,  Liouville theorem for Dunkl polyharmonic functions. SIGMA Symmetry Integrability Geom. Methods Appl. \textbf{4}: 6  (2008)


\bibitem{Rosler-2003}M. R\"osler,  A positive radial product formula for the Dunkl kernel. Trans. Amer. Math. Soc. \textbf{355}:2413-2438 (2003). 

\bibitem{Rosler-2008}M. R\"osler,  M. Voit,  Dunkl theory, convolution algebras, and related Markov processes, in Harmonic and stochastic analysis of Dunkl processes. P. Graczyk, M. R\"osler, M. Yor (eds.), 1–112, Travaux en cours 71, Hermann, Paris, (2008).

\bibitem{Seregin}G.A. Ser\"egin, A local estimate of the Caccioppoli inequality type for extremal variational problems of Hencky plasticity. Akad. Nauk SSSR Sibirsk. Otdel. Inst. Mat. Novosibirsk \textbf{145}:127–138 (1988).

\bibitem{Skrzypczak}I. Skrzypczak, Hardy inequalities resulted from nonlinear problems dealing with $A$-Laplacian. NoDEA Nonlinear Differential Equations Appl. \textbf{21}(6):841-868 (2014).

\bibitem{Velichu}A. Velicu,  Sobolev-type inequalities for Dunkl operators. J. Funct. Anal. \textbf{279}(7):108695 (2020).

\bibitem{Troianiello}G.M. Troianiello,   Estimates of the Caccioppoli-Schauder type in weighted function spaces. Trans. Am. Math. Soc. \textbf{334}(2): 551-573 (1992)..


\bibitem{Xing}Y. Xing,  S. Ding,   Caccioppoli inequalities with Orlicz norms for solutions of harmonic equations and applications. Nonlinearity, \textbf{23}(5):1109 (2010).
\end{thebibliography}
 
\end{document}